\documentclass[leqno,a4paper,10pt]{article}

\usepackage[utf8x]{inputenc}
\usepackage{amsfonts}
\usepackage{amsmath}\usepackage{amssymb}
\usepackage{amsthm}
\usepackage{mathrsfs}
\usepackage{latexsym}
\usepackage{graphicx}
\usepackage{bm}
\usepackage{inputenc}
\usepackage{enumitem}

\usepackage{hyperref}
\usepackage{graphicx}

\swapnumbers

\newtheorem{Lemma}{Lemma}[section]

\newtheorem{Theorem}[Lemma]{Theorem}

\theoremstyle{definition}

\newtheorem{Definition}[Lemma]{Definition}

\theoremstyle{remark}

\newtheorem{Remark}[Lemma]{Remark}

\newtheoremstyle{citing}
{3pt}
{3pt}
{\itshape}
{}
{\bfseries}
{.}
{.5em}
{\thmnote{#3}}
\theoremstyle{citing}

\newtheoremstyle{proof*}
{3pt}
{3pt}
{\rmfamily}
{}
{ \itshape}
{.}
{.5em}
{\thmnote{#3}}
\theoremstyle{proof*}
\newtheorem*{proof*}{}

\DeclareMathOperator{\restrict}{\llcorner}

\DeclareMathOperator{\Tan}{Tan}     
     
\DeclareMathOperator{\im}{im}       
\DeclareMathOperator{\Lip}{Lip}     
\DeclareMathOperator{\grad}{grad}  
\DeclareMathOperator{\dmn}{dmn}     
\DeclareMathOperator{\Nor}{Nor}
     
\DeclareMathOperator{\Der}{D}       
       
\DeclareMathOperator{\ap}{ap}

\DeclareMathOperator{\reach}{reach}

\DeclareMathOperator{\gr}{gr}

\DeclareMathOperator{\Dis}{Dis}

\newcommand{\Real}[1]{ \mathbf{R}^{#1}}
\newcommand{\Haus}[1]{ \mathscr{H}^{#1} }
\newcommand{\Leb}[1]{ \mathscr{L}^{#1} }
\newcommand{\rect}[1]{(\mathscr{H}^{#1},#1)}

\newcommand{\Hdensity}[3]{\bm{\Theta}^{#1}(\mathscr{H}^{#1}\restrict \,#2,#3 )}
\newcommand{\supLdensity}[3]{\bm{\Theta}^{*#1}(\mathscr{L}^{#1}\restrict \,#2,#3 )}
\newcommand{\Ldensity}[3]{\bm{\Theta}^{#1}(\mathscr{L}^{#1}\restrict \,#2,#3 )}

\title{Fine properties of the curvature of arbitrary closed sets}
\author{Mario Santilli\footnote{email: \texttt{mario.santilli@math.uni-augsburg.de}}\\ Augsburg Universit\"at}

\begin{document}
	\maketitle
	
	\begin{abstract}
Given an arbitrary closed set $ A $ of $ \Real{n} $, we establish the relation between the eigenvalues of the approximate differential of the spherical image map of $ A $ and the principal curvatures of $ A $ introduced by Hug-Last-Weil, thus extending a well known relation for sets of positive reach by Federer and Z\"ahle. Then we provide for every $ m = 1, \ldots , n-1 $ an integral representation for the support measure $ \mu_{m} $ of $ A $ with respect to the $ m $ dimensional Hausdorff measure. 

Moreover a notion of second fundamental form $Q_{A}$ for an arbitrary closed set $ A $ is introduced so that the finite principal curvatures of $ A $ correspond to the eigenvalues of $ Q_{A} $. We prove that the approximate differential of order $ 2 $, introduced in a previous work of the author, equals in a certain sense the absolutely continuous part of $ Q_{A} $, thus providing a natural generalization to higher order differentiability of the classical result of Calderon and Zygmund on the approximate differentiability of functions of bounded variation.
	\end{abstract}

\paragraph{\small MSC-classes 2010.}{\small 52A22, 53C65 (Primary); 28A75, 60D05 (Secondary).}

\section{Introduction}

\subsection*{Background}
The theory of curvature of arbitrary closed subsets of the Euclidean space, which finds its roots in the landmark paper of Federer \cite{MR0110078} on sets of positive reach, has been initiated by Stach\'o in \cite{MR534512} and continued by Hug-Last-Weil in \cite{MR2031455}. If $ A \subseteq \Real{n} $ is a closed set and $\bm{\delta}_{A}$ is the distance function from $ A $, these authors introduced the \emph{generalized normal bundle of $ A $}, 
\begin{equation*}
N(A) = (A \times \mathbf{S}^{n-1}) \cap \{ (a,u) :  \bm{\delta}_{A}(a+su)=s \; \textrm{for some $ s > 0 $}\}
\end{equation*}
and they observed that there exists a countable collection $A_{1},A_{2}, \ldots $ of closed sets of positive reach and compact boundary such that
\begin{equation*}
N(A) \subseteq \bigcup_{i=1}^{\infty}N(A_{i}).
\end{equation*}
On the basis of this fact, it follows that $N(A)$ is a countably $n-1$ rectifiable subset of $ \Real{n} \times \mathbf{S}^{n-1} $ and its $n-1$ dimensional approximate tangent plane coincides with that of one of the sets $ A_{n} $ at $ \Haus{n-1} $ a.e.\ $ (a,u) \in N(A) $. This observation allows to introduce the \emph{principal curvatures of $ A $}, 
\begin{equation}\label{principal curvatures HLW}
-\infty< \lambda_{A,1}(a,u) \leq \ldots \leq \lambda_{A,n-1}(a,u) \leq \infty,
\end{equation}
at $ \Haus{n-1} $ a.e.\ $ (a,u) \in N(A) $, using the notion of principal curvature for sets of positive reach introduced by Z\"ahle in \cite{MR849863}. The support measures $ \mu_{0}, \ldots ,\mu_{n-1} $ of $ A $ are then introduced by
\begin{equation}\label{explicit formula for support measures}
\mu_{i}(D) = \frac{1}{(n-i)\bm{\alpha}(n-i)} \int_{D}H_{n-i-1}d\Haus{n-1},
\end{equation}
whenever $ D \subseteq \Real{n} \times \mathbf{S}^{n-1} $ is an $ \Haus{n-1} $ measurable set such that the integral on the right side exists (finite or infinite). Here $ H_{j} $ denotes \emph{the $ j $-th symmetric function of the principal curvatures of $ A $},
\begin{equation}\label{symmetric functions}
H_{j} =  \sum_{\{l_{1}, \ldots , l_{j}\} \subseteq \{1, \ldots , n-1\}} \;  \bigg(\prod_{i=1}^{n-1}(1 + \lambda_{A,i}^{2})^{-1/2}\bigg)\prod_{i=1}^{j}\lambda_{A,l_{i}}.
\end{equation}
The main result of the theory, the \emph{Steiner formula} in \cite[Theorem 2.1]{MR2031455}, is phrased in terms of these support measures; see also \cite[Theorem 1]{MR3153586} where a corrected version of this formula is pointed out. Despite this important result, several basic questions in the theory remain undisclosed and it is our aim in this work to investigate some of them.

The theory of curvature for arbitrary closed sets has found applications so far in the study of random closed sets in stochastic geometry (see \cite[sections 7-8]{MR2031455}, \cite{MR2307063}) and in spatial statistics (see \cite{MR3153586}). On the other hand, the fact that this is a theory developed with no a priori assumptions on the structure of the sets (e.g.\ convex, positive reach, etc..), makes it certainly appealing in the study of singular surfaces arising as solutions of variational problems (e.g.\ varifolds). We will present these applications in subsequent works. 

\subsection*{Results of the present paper}
\paragraph{Relating the principal curvatures to the eigenvalues of the differential of the spherical image map.}
If $ A \subseteq \Real{n} $ is a closed set, let $ \bm{\xi}_{A} $ be the nearest point projection onto $ A $ and let $\bm{\nu}_{A} $ be the spherical image map, i.e.\ $\bm{\nu}_{A}(x) = \bm{\delta}_{A}(x)^{-1}(x-\bm{\xi}_{A}(x))$ for $ x \in \dmn \bm{\xi}_{A} \sim A $. If $ A $ is a set of positive reach then it is well known (Federer \cite[4.8]{MR0110078} and Z\"ahle \cite{MR849863}) that $ \bm{\nu}_{A} $ is differentiable with symmetric differential at $ \Leb{n} $ a.e.\ $ x \in \{ y : 0 < \bm{\delta}_{A}(y) < \reach(A)  \} $ and, denoting by $ \chi_{A,1}(x) \leq \ldots \leq \chi_{A,n-1}(x) $ the eigenvalues of $ \Der\bm{\nu}_{A}(x)|\{v : v \bullet \bm{\nu}_{A}(x) =0   \} $, it follows that the principal curvature $ \lambda_{A,i}(a,u) $ of $ A $ at $ \Haus{n-1} $ a.e.\ $ (a,u)\in N(A) $ is given by
\begin{equation}\label{relation between principal curvatures: positive reach}
\lambda_{A,i}(a,u)= \frac{\chi_{A,i}(a+ru)}{1-r\chi_{A,i}(a+ru)} \quad \textrm{for $ 0 < r < \reach(A) $;}
\end{equation}
in fact the right side does not depend on $ r \in (0, \reach(A)) $. On the other hand, it has been proved in \cite{MR0310150} that if $ A $ is an arbitrary closed subset of $ \Real{n} $ then a certain extension of the nearest point projection $ \bm{\xi}_{A} $ on $ \Real{n}$ is differentiable with symmetric differential at $ \Leb{n} $ almost every $ x \in \Real{n} $ (the nearest point projection is not well defined on $ \Real{n} $ unless $ A $ is convex). Therefore it is a natural question to understand if the principal curvatures of an arbitrary closed set introduced in \cite{MR2031455} can be realized by mean of a suitable extension of \eqref{relation between principal curvatures: positive reach}. We provide the answer in sections \ref{section: approx. diff of n.p.p.} and \ref{section: s.f.f.}, whose content we now briefly describe. The main purpose of section \ref{section: approx. diff of n.p.p.} is to analyse the set of approximate differentiability points of $ \bm{\xi}_{A} $ for an arbitrary closed set $ A $ and to describe the tangential and curvature properties of the level sets $ S(A,r)$ of the distance function $ \bm{\delta}_{A} $ in terms of $ \bm{\nu}_{A} $ and its approximate differential, see \ref{fine properties of level sets}. This is done introducing a meaningful reach-type function $ \rho(A, \cdot) $ in \ref{definition of rho} and analysing the behaviour of $ \bm{\xi}_{A} $ on the super-level sets of $ \rho(A, \cdot) $, see \ref{A lambda}. A first consequence of this analysis is contained in \ref{regular points}-\ref{approximate diff of nearest point pr}, where we provide a refined version of the differentiability theorem in \cite{MR0310150}. As a second consequence, we obtain in section \ref{section: s.f.f.} the answer to the question posed at the beginning of this paragraph, which we summarize in the following theorem.

\begin{Theorem}\label{relating principal curvatures}
	If $ A \subseteq \Real{n} $ is a closed set then $ \bm{\nu}_{A} $ is approximately differentiable with symmetric approximate differential at $ \Leb{n} $ a.e.\ $ x \in \Real{n} \sim A $ and, denoting by $ \chi_{A,1}(x) \leq \ldots \leq \chi_{A,n-1}(x) $ the eigenvalues of $ \ap\Der\bm{\nu}_{A}(x)|\{v : v \bullet \bm{\nu}_{A}(x) =0   \} $,
	\begin{equation*}
		\lambda_{A,i}(a,u)= \frac{\chi_{A,i}(a+ru)}{1-r\chi_{A,i}(a+ru)}
	\end{equation*}
	for $ \Haus{n-1} $ a.e.\ $ (a,u) \in N(A) $, for every $ 0 < r < \sup\{ s : \bm{\delta}_{A}(a+su) =s   \} $ and $ i = 1, \ldots , n-1 $.
\end{Theorem}
The number $\sup\{ s : \bm{\delta}_{A}(a+su) =s   \}$ equals the \emph{reach function of $ A $} at $(a,u)$ introduced in \cite[p.\ 292]{MR3153586} and it naturally appears in the Steiner formula (see \cite[Theorem 1]{MR3153586}). Moreover we introduce a symmetric bilinear form (which we call \emph{second fundamental form of $ A $ at $ a $ in the direction $ u $})
\begin{equation}\label{intro: sff}
Q_{A}(a,u): T_{A}(a,u) \times T_{A}(a,u) \rightarrow \Real{},
\end{equation}
at $ \Haus{n-1} $ a.e.\ $ (a,u)\in N(A) $, whose eigenvalues coincide with the finite principal curvatures of $ A $. Here $ T_{A}(a,u) $ is a linear subspace of $ \Real{n} $ whose dimension can vary from $ 0 $ to $ n-1 $. The second fundamental form will be further investigated in sections \ref{section: curvature and strata} and \ref{section: relation with second order rectifiability}.

\paragraph{Integral representation of the support measures.}In section \ref{section: curvature and strata} we consider the following natural stratification of a closed set $ A $: for each $ 0 \leq m \leq n $, we define \textit{the $m$-th stratum of $A$} as
\begin{equation*}
A^{(m)} = A \cap \{ a : \dim \bm{\xi}_{A}^{-1}\{a\} = n-m \} =  A \cap \{ a :   0 < \Haus{n-m-1}(N(A,a)) < \infty \}
\end{equation*}
(recall that $ \bm{\xi}_{A}^{-1}\{a\}$ is a convex set for every $ a \in A $). The structure of this stratification has been investigated in \cite{2017arXiv170309561M}, where it is proved (notice \ref{characterization of the strata in terms of Normal bundle}) that $ A^{(m)} $ is always countably $ (\Haus{m},m) $ rectifiable of class $ 2 $, see \cite[4.12]{2017arXiv170309561M}. The main point here is to analyse the behaviour of the principal curvatures of $ A $ on each strata, see \ref{Area formula Gauss map:remark} and \ref{representation of support measures}\eqref{representation of support measures:2}. Then for each integer $ 1 \leq m \leq n-1 $ we obtain the following integral representation formula of the support measure $ \mu_{m} $ with respect to the $ m $ dimensional Hausdorff measure $ \Haus{m} $. For arbitrary closed sets this result appears to be known only if $ m = n-1 $, see \cite[4.1]{MR2031455} (see also \cite[5.5]{MR1742247} for the special case of sets of positive reach).
\begin{Theorem}(see \ref{representation of support measures})
If $ A \subseteq \Real{n} $ is a closed set, $ \mu_{0}, \ldots , \mu_{n-1} $ are the support measures of $ A $, $ 1 \leq m \leq n-1 $ is an integer, $ S $ is a countable union of Borel subsets with finite $ \Haus{m} $ measure and $ T \subseteq N(A)|S $ is $ \Haus{n-1} $ measurable then
	\begin{equation*}
	\mu_{m}(T) = \frac{1}{(n-m)\bm{\alpha}(n-m)}\int_{A^{(m)}} \Haus{n-m-1}\{v : (z,v) \in T \}d\Haus{m}z.
	\end{equation*}
\end{Theorem}

\paragraph{Second order approximate differentiability.} Finally in section \ref{section: relation with second order rectifiability} we analyse the relation of the present notion of curvature with the notion of approximate curvature for second-order rectifiable sets introduced by the author in \cite{2017arXiv170107286S}. In the latter, second order rectifiable sets are characterized by the existence of the approximate differential of order $ 2 $ at almost every point (we refer to \cite[1.2]{2017arXiv170107286S} for a precise statement, which actually holds for all possible orders of rectifiability). In this section we complement this characterization with the following result:
\begin{Theorem}\label{agreement with classical second fundamental form: intro}(see \ref{ap diff for sets}, \ref{ap differentials} and \ref{agreement with classical second fundamental form})
	Let $ A \subseteq \Real{n} $ be a closed set, $ 1 \leq m \leq n-1 $  and let $ S \subseteq A $ be $ \Haus{m} $ measurable and $ \rect{m} $ rectifiable of class $ 2 $. Then there exists $ R \subseteq S $ such that $ \Haus{m}(S \sim R) = 0 $ and 
	\begin{equation*}
	\textstyle	\ap \Tan(S,a) = T_{A}(a,u) \quad \ap \Der^{2}S(a)(\tau, \upsilon) \bullet u = - Q_{A}(a,u)(\tau, \upsilon) 
	\end{equation*}
	for every $ \tau, \upsilon \in T_{A}(a,u) $ and for $ \Haus{n-1} $ a.e.\ $ (a,u) \in N(A)|R $.
\end{Theorem}

In other words this theorem asserts that "the absolutely continuous part of the second fundamental form $ Q_{A} $, when restricted on a second order rectifiable subset $ S $ of $ A $, coincides with the approximate differential of order $ 2 $ of $ S $". This result has an interesting analogy with the classical theorem of Calderon and Zygmund asserting that the absolutely continuous part of the total differential of a function of bounded variation coincides with its approximate gradient. This analogy is further strengthened if we look at the primitive $ g $ of the Cantor function $ f $ ($ f $ is a function of bounded variation whose total differential cannot be fully described by the approximate derivative), see \ref{Cantor function}. The epigraph of $ g $ is a closed convex set $ A $ of $ \Real{2} $ which admits a subset $ T \subseteq \partial A $ such that $ \Haus{1}(N(A)|T) > 0 $ and 
\begin{equation*}
	T_{A}(a,u) = \{0\} \quad \textrm{for $ \Haus{1} $ a.e.\ $ (a,u) \in N(A)|T $.}
\end{equation*}
It follows that the second fundamental form cannot be fully described by the approximate differential of order $ 2 $. 

\paragraph{Acknowledgements.} The author thanks Prof.\ Ulrich Menne for conversations concerning the content of section \ref{section: relation with second order rectifiability}; moreover the author thanks the referee for his or her careful reading of the manuscript.

\section{Preliminaries}\label{section: Preliminaries}

The notation and the terminology agree with \cite[pp.\ 669--676]{MR0257325}. Let $ m $ be a non negative integer. The symbols $ \mathbf{U}(a,r) $ and $ \mathbf{B}(a,r) $ denote the open and closed ball with centre $ a $ and radius $ r $ (\cite[2.8.1]{MR0257325}); $ \mathbf{S}^{m} $ is the $ m $ dimensional unit sphere in $ \Real{m+1} $ (\cite[3.2.13]{MR0257325}); $ \Leb{m} $ and $ \Haus{m} $ are the $ m $ dimensional Lebesgue and Hausdorff measure (\cite[2.10.2]{MR0257325}); $ \bm{\alpha}(m) = \Leb{m}(\mathbf{U}(0,1)) $; given a measure $ \mu $, we denote by $ \bm{\Theta}^{\ast m}(\mu, \cdot) $, $ \bm{\Theta}_{\ast}^{m}(\mu, \cdot) $ and $ \bm{\Theta}^{m}(\mu, \cdot) $ the $ m $ dimensional densities of $ \mu $ (\cite[2.10.19]{MR0257325}); $ \mathbf{G}(m,k) $ is the Grassmann manifold of all $ k $ dimensional subspaces in $ \Real{m} $ (\cite[1.6.2]{MR0257325}). Moreover, given a function $ f $, we denote by $ \dmn f $ and $ \im f $ the domain and the image of $ f $. The symbol $ \bullet $ denotes the standard inner product of $\Real{m}$. If $T$ is a linear subspace of $\Real{m}$, then $T_{\natural} : \Real{m} \rightarrow \Real{m}$ is the orthogonal projection onto $T$ and
\begin{equation*}
	T^{\perp} = \Real{m} \cap \{ v : v \bullet u =0 \; \textrm{for $u \in T$} \}.
\end{equation*}
If $X$ and $Y$ are sets, $ Z \subseteq X \times Y $ and $S \subseteq X$, then
\begin{equation*}
	Z | S = Z \cap \{ (x,y) : x \in S  \}.
\end{equation*}
The maps $ \mathbf{p}, \mathbf{q} : \Real{m} \times \Real{m} \rightarrow \Real{m} $ are defined by
\begin{equation*}
	\mathbf{p}(x,v)= x, \quad \mathbf{q}(x,v) = v.
\end{equation*}
To treat the classical concept of rectifiable sets we adopt the terminology introduced in \cite[3.2.14]{MR0257325}. Moreover, if $A \subseteq \Real{n}$ we say that \textit{$A$ is countably $\rect{m}$ rectifiable of class $2$} if $A$ can be $\Haus{m}$ almost covered by the union of countably many $m$ dimensional submanifolds of class $2$ of $\Real{n}$; we omit the prefix ``countably'' when $\Haus{m}(A)<\infty$. We refer to \cite[3.1.21]{MR0257325} for the notions of \emph{tangent and normal cone of a set}; moreover, given a measure $ \mu $ and a positive integer $ m $, the \emph{approximate tangent cone $ \Tan^{m}(\mu, \cdot) $} is defined as in \cite[3.2.16]{MR0257325}. Finally, if $X$ and $Y$ are metric spaces and $f : X \rightarrow Y$ is a function such that $f$ and $f^{-1}$ are Lipschitzian functions, then we say that $f$ is a \emph{bi-Lipschitzian homeomorphism.}

\subsection*{Second fundamental form and normal bundle of submanifolds of class $ 2 $}

\begin{Definition}
	Suppose $ 1 \leq m \leq n $ are integers, $ M $ is an $ m $ dimensional submanifold of class $ 2 $ of $ \Real{n} $ and $ a \in M $. Then we call second fundamental form of $ M $ at $ a $ the unique symmetric $ 2 $ linear function
	\begin{equation*}
	\mathbf{b}_{M}(a) : \Tan(M,a) \times \Tan(M,a) \rightarrow \Nor(M,a)
	\end{equation*}
	such that $ \mathbf{b}_{M}(a)(u,v) \bullet \nu(a) = - \Der \nu(a)(u) \bullet v $ for each $ u, v \in \Tan(M,a) $, whenever $ \nu : M \rightarrow \Real{n} $ is of class $ 1 $ relative to $ M $ with $ \nu(x) \in \Nor(M,x) $ for every $ x \in M $.
\end{Definition}

The following lemma is well known in differential geometry.

\begin{Lemma}\label{normal bundle of smooth submanifolds}
	Let $ M \subseteq \Real{n} $ be an $ m $ dimensional submanifold of class $ 2 $ and let $N=\Nor(M) \cap (M \times \mathbf{S}^{n-1}) $.
	
	Then $N$ is an $n-1$ dimensional submanifold of class $1$ of $\Real{n} \times \Real{n} $ and, if $(a,u)\in N$ then $\Tan(N,(a,u))$ is the set of $(\tau, v + \Der \nu(a)(\tau) )$ such that \mbox{$\tau \in \Tan(M,a)$,} $v \in \Nor(M,a)$ is orthogonal to $u$ and $ \nu $ is a unit normal vector field of class $1$ on an open neighborhood of $a$ such that $\nu(a)=u$.
\end{Lemma}

\begin{proof}
	The conclusion is a direct consequence of the fact that, using a normal frame of $M$ in an open neighborhood $Z$ of $a$, we can locally parametrize $N$ at $(a,u)$ using the product manifold $(M \cap Z) \times \mathbf{S}^{n-m-1}$. 
\end{proof}
\begin{Remark}\label{remark on s.f.f for smooth submanifolds}
	If $(a,u)\in N$, $\tau \in \Tan(M,a)$, $\tau_{1} \in \Tan(M,a)$ and $\sigma_{1} \in \Real{n}$ is such that $(\tau_{1}, \sigma_{1}) \in \Tan(N,(a,u)) $, then
	\begin{equation*}
	\tau \bullet \sigma_{1} = - \mathbf{b}_{M}(a)(\tau, \tau_{1}) \bullet u.
	\end{equation*}
\end{Remark}

\subsection*{Approximate differentiability for functions and sets}

First we recall the following measure-theoretic notions of limit and differentiability for functions, which play a key role in section \ref{section: approx. diff of n.p.p.}.

\begin{Definition}\label{approximate limits}
	Let $ f $ be a function mapping a subset of $ \Real{n} $ into some set $ Y $ and let $ a \in \Real{n} $. If $ Y $ is a normed vector space, a point $ y \in Y $ is \emph{the approximate limit of $ f $ at $ a $} if and only if
	\begin{equation*}
	\Ldensity{n}{\Real{n} \sim \{ x : | f(x) - y |  \leq \epsilon \}}{a} = 0 \quad \textrm{for every $ \epsilon > 0 $}
	\end{equation*}
	and we denote it by $ \ap\lim_{x \to a} f(x) $. If $ Y = \overline{\Real{}} $, a point $ t \in \overline{\Real{}} $ is \emph{the approximate lower limit of $ f $ at $ a $} [\emph{the approximate upper limit of $ f $ at $ a $}] if and only if
	\begin{equation*}
	t = \sup\{ s: \Ldensity{n}{ \{ x :  f(x) < s \}}{a} = 0 \}
	\end{equation*}
	\begin{equation*}
	\big[t = \inf\{ s: \Ldensity{n}{ \{ x :  f(x) > s \}}{a} = 0 \}\big]
	\end{equation*}
	and we denote it by $ \ap \liminf_{x \to a} f(x) $ [$ \ap \limsup_{x \to a} f(x) $].
\end{Definition}

\begin{Definition}\label{approximate differentiability for functions}
	Let $ n \geq 1 $, $ \nu \geq 1 $ and $ k \geq 0 $ be integers, $ A \subset \mathbf{R}^{n} $, $ f : A \rightarrow \Real{\nu} $ and $ a \in \mathbf{R}^{n} $.
	
	We say that $ f $ is \textit{approximately differentiable of order $ k $ at $ a $} if there exists a polynomial function $ P :\Real{n} \rightarrow \Real{\nu} $ of degree at most $ k $ such that $ P(a)= f(a) $ if $ a \in A $, and
	\begin{equation*}
	\ap\lim_{x \to a} \frac{|f(x)-P(x)|}{|x-a|^{k}}=0.
	\end{equation*}
	We let $ \ap \Der^{i}f(a) = \Der^{i}P(a) $ for $ i = 1, \ldots , k $.
\end{Definition}

\begin{Remark}\label{approximate differentiability for functions remark}
The following statement follows immediately from \ref{approximate limits} and \ref{approximate differentiability for functions}. \emph{Suppose $ n,\nu, k, A, f, a $ are as in \ref{approximate differentiability for functions} and $ B \subseteq A $. Then $ f|B $ is approximately differentiable of order $ k $ at $ a $ if and only if $ f $ is approximately differentiable of order $ k $ at $ a $ and $ \Ldensity{n}{\Real{n} \sim B}{a} =0 $. In this case $ \ap \Der^{i}(f|B)(a) = \ap \Der^{i}f(a) $ for $ i = 1, \ldots , k $.}
\end{Remark}

We recall now from \cite[3.8, 3.19, 3.20]{2017arXiv170107286S} the notion of approximate differentiability for sets.
\begin{Definition}\label{ap diff for sets}
	Let $ n \geq 1 $ and $ k \geq 1 $ be integers, $ A \subseteq \Real{n} $, $ a \in \Real{n} $. We say that $ A $ is \textit{approximately differentiable of order $ k $ at $ a $} if and only if there exist an integer $ 1 \leq m \leq n $, \mbox{$ T \in \mathbf{G}(n,m) $} and a polynomial function $ P : T \rightarrow T^{\perp} $ of degree at most $ k $ such that $ P(0) = 0 $, $ \Der P(0) = 0 $ and the following two conditions hold:
	\begin{enumerate}
		\item \label{ap diff for sets: dimension} for every $ \epsilon> 0 $ there exists $ \eta > 0 $ such that 
		\begin{equation*}
		\Haus{m}(\mathbf{B}(z, \epsilon r) \cap \{ x-a : x \in A  \}) \geq \eta r^{m}
		\end{equation*}
		for every $ z \in T \cap \mathbf{B}(0,r) $ and $ 0 \leq r \leq \eta $,
		\item \label{ap diff for sets: approx by polynomials} for every $ \epsilon > 0 $,
		\begin{equation*}
		\lim_{r \to 0} \frac{ \Haus{m}\left( \{ x-a : x \in A  \} \cap \mathbf{B}(0,r) \cap \{ z: \bm{\delta}_{\gr(P)}(z) > \epsilon\,r^{k} \} \right) }{\bm{\alpha}(m)r^{m}  }  =0,
		\end{equation*}
		where $ \gr P = \{ \chi + P(\chi) : \chi \in T \} $.
	\end{enumerate}
\end{Definition}

\begin{Definition}\label{ap differentials}
	Let $ n $, $ k $, $ A $, $ a $, $ m $, $ T $ and $ P $ as in \ref{ap diff for sets}. Then we define 
	\begin{equation*}
	\ap \Tan(A,a) = T, \quad \ap \Nor(A,a) = T^{\perp},
	\end{equation*}
	\begin{equation*}
	\ap \Der^{k}A(a) = \Der^{k}(P \circ T_{\natural})(0).
	\end{equation*} 
\end{Definition}

\begin{Remark}
One can prove, using a standard density-argument, that if $ M $ is an $ m $ dimensional submanifold of class $ 1 $ [class $ 2 $] in $ \Real{n} $ and $ A \subseteq M $ is $ \Haus{m} $ measurable with $ \Haus{m}(A)< \infty $, then 
	\begin{equation*}
	\Tan(M,a) = \ap \Tan(A,a) \quad \textrm{for $ \Haus{m} $ a.e.\ $ a \in A $}
	\end{equation*}
	\begin{equation*}
	[\ap \Der^{2}A(a)|\ap\Tan(A,a)\times \ap \Tan(A,a) = \mathbf{b}_{M}(a) \quad \textrm{for $ \Haus{m} $ a.e.\ $ a \in A $.}]
	\end{equation*}
\end{Remark}

\begin{Remark}
For a set $ A \subseteq \Real{n} $ other notions of measure-theoretic tangent planes are well known, see \cite[1.3, 1.4]{2017arXiv170107286S}. If $ A $ is $ \Haus{m} $ measurable and $ \Haus{m}(A) < \infty $ then the sets of points where these tangent planes exist and belong to $ \mathbf{G}(n,m) $ are $ \Haus{m} $ almost equal to the set of points where $ \ap \Tan(A, \cdot) $ exists and belongs to $ \mathbf{G}(n,m) $.
\end{Remark}

\begin{Remark}
	A characterization of higher order rectifiable sets is obtained in \cite[3.23, 5.6]{2017arXiv170107286S} in terms of the approximate differentiability given in \ref{ap diff for sets}.
\end{Remark}

\subsection*{Level sets of distance function}

\begin{Definition}\label{definition: distance function}
	Let $ A \subseteq \Real{n} $ be a closed set. We define
	\begin{equation*}
	\bm{\delta}_{A}(x) = \inf\{|x-a| : a \in A   \} \quad \textrm{for $ x \in \Real{n} $,}
	\end{equation*}
	\begin{equation*}
	S(A,r) = \{ x : \bm{\delta}_{A}(x) = r  \} \quad \textrm{for $ r > 0 $.}
	\end{equation*}
\end{Definition}


In this paper we need the following result on the rectifiability properties of the level sets of $ \bm{\delta}_{A} $.


\begin{Theorem}\label{structure of level sets}
	Let $ A $ be a closed subset of $ \Real{n} $ and $ r > 0 $. 
	\begin{enumerate}
		\item\label{structure of level sets:1} If $ K \subseteq \Real{n} $ is compact then $ S(A,r) \cap K $ is $n-1$ rectifiable.
		\item\label{structure of level sets:3} $ S(A,r)$ is countably $ \rect{n-1} $ rectifiable of class $ 2 $.
	\end{enumerate}
\end{Theorem}

\begin{proof}
	If $ A $ is bounded then the proof of \eqref{structure of level sets:1} is contained in \cite[2.3]{MR2865426} (which relies on \cite{MR816398}). If $ A $ is unbounded then the proof can be readily reduced to the previous case noting that if $ r > 0 $ and $ K \subseteq \Real{n} $ is compact then the set
	\begin{equation*}
	C = \bigcup_{x \in S(A,r) \cap K} A \cap \{ a : |x-a|= \bm{\delta}_{A}(x)   \}
	\end{equation*}
	is compact and $ S(A,r) \cap K \subseteq S(C,r) $.
	
	We notice that for each $ x \in S(A,r) $ there exists $ v \in \Real{n} \sim \{0\} $ such that $ \mathbf{U}(x+v, |v|) = \varnothing $. In fact, we can choose $ v = a - x $ for $ a \in A $ such that $ |x-a| = r $. Therefore \eqref{structure of level sets:3} comes from \cite[4.12]{2017arXiv170309561M}. Notice that \cite[4.12]{2017arXiv170309561M} also implies that $S(A,r)$ is countably $ n-1 $ rectifiable, a piece of information already contained in \eqref{structure of level sets:1}.
\end{proof}

\begin{Remark}
The local structure of the level sets of the distance function has been thoroughly studied in the last decades; see \cite{MR0413112}, \cite{MR0287442}, \cite{MR816398} and \cite{MR2954647}. However, here we only use the rectifiability properties in \ref{structure of level sets}.
\end{Remark}

\begin{Definition}
	If $ A \subseteq \Real{n} $ is a closed set, we define \emph{the positive boundary} $ \partial^{+}A $ of $ A $ as the set of all $ x \in A $ such that there exists $ v \in \Real{n} \sim \{0\} $ with $ A \cap \mathbf{U}(x+v, |v|) =\varnothing $. 
\end{Definition}

The following result is contained in \cite[2.5]{MR2865426} when $ A $ is a compact set. 

\begin{Lemma}\label{Extension of a lemme of Rataj-Winter}
Let $ A \subseteq \Real{n} $ be a closed set and let $ P_{r} = \{ x : \bm{\delta}_{A}(x) \leq r  \} $ for $ r > 0 $. Then for all $ r > 0 $ up to a countable set, 
\begin{equation*}
	\Haus{n-1}(S(A,r) \sim \partial^{+} P_{r}  ) =0.
\end{equation*}
\end{Lemma}

\begin{proof}
	If $ r > 0 $ and $ i \geq 1 $ is an integer, we define $P_{i,r} = \{ x : \bm{\delta}_{A \cap \mathbf{B}(0,i)}(x)  \leq r  \}$. We fix two integers $ i \geq 1 $ and $ j \geq 1 $ and we prove that for all $ 0 < r < j $ up to a countable set,
	\begin{equation*}
	\Haus{n-1}(S(A,r) \cap \mathbf{U}(0,i) \sim \partial^{+}P_{r}) =0.
	\end{equation*}
Let $ 0 < r < j $ and $ x \in S(A,r) \cap \mathbf{U}(0,i) \cap \partial^{+}P_{i+j,r} $. Then there exist $ s > 0 $ and $ v \in \Real{n} $ with $ |v|=1 $ and $ \mathbf{U}(x+sv,s) \cap P_{i+j,r}= \varnothing $. Evidently we can choose $ s $ small so that $ \mathbf{U}(x+sv,s) \subseteq \mathbf{U}(0,i) $. If there was $ z \in \mathbf{U}(x+sv,s) $ such that $ \bm{\delta}_{A}(z) \leq r $ then we could choose $ a \in A $ so that $ |z-a| = \bm{\delta}_{A}(z) $ and infer that 
\begin{equation*}
a \in A \cap \mathbf{B}(0, i+j), \quad \bm{\delta}_{A \cap \mathbf{B}(0, i+j)}(z) \leq r,
\end{equation*}
whence we would get a contradiction. Therefore
\begin{equation*}
S(A,r) \cap \mathbf{U}(0,i) \cap \partial^{+}P_{i+j,r} \subseteq S(A,r) \cap \mathbf{U}(0,i) \cap \partial^{+}P_{r}.
\end{equation*}
Moreover we observe that
\begin{equation*}
	S(A,r) \cap \mathbf{U}(0,i) \subseteq S(A \cap \mathbf{B}(0,i+j),r) \quad \textrm{for all $ 0 < r < j $.}
\end{equation*}
Now we employ \cite[2.5]{MR2865426} to infer
\begin{equation*}
\Haus{n-1}(S(A,r) \cap \mathbf{U}(0,i) \sim \partial^{+}P_{r}  ) =0
\end{equation*}
for all $ 0 < r < j $, up to a countable set.
\end{proof}

\section{Fine properties of the nearest point projection}\label{section: approx. diff of n.p.p.}

The main objective of this section is to analyse the fine properties of the nearest point projection $ \bm{\xi}_{A} $ and relate them to the tangential and curvature properties of the distance sets $S(A,r)$.
  
We start introducing some basic notation. It will be repeatedly used through the rest of this paper together with the notation already introduced in \ref{definition: distance function}.
	
	\begin{Definition}[Basic notation]\label{nearest point projection}
		Suppose $ A \subseteq \Real{n} $ is closed and $U$ is the set of all $x \in \Real{n}$ such that there exists a unique $a \in A$ with $|x-a| = \bm{\delta}_{A}(x)$. The \textit{nearest point projection onto~$A$} is
		the map $\bm{\xi}_{A}$ characterised by the requirement
		\begin{equation*}
		| x- \bm{\xi}_{A}(x)| = \bm{\delta}_{A}(x) \quad \textrm{for $x \in U$}.
		\end{equation*}
		\noindent Let $ \bm{\nu}_{A} $ and $ \bm{\psi}_{A} $ be the functions on $U \sim A$ such that 
	\begin{equation*}
	\bm{\nu}_{A}(z) = \bm{\delta}_{A}(z)^{-1}(z -  \bm{\xi}_{A}(z)) \quad \textrm{and} \quad \bm{\psi}_{A}(z)= (\bm{\xi}_{A}(z), \bm{\nu}_{A}(z)),
	\end{equation*}
	whenever $ z \in U \sim A $. We refer to $\bm{\nu}_{A}$ as \emph{the spherical image map of $ A $}. Finally, 
		\begin{equation*}
		U(A) = \dmn \bm{\xi}_{A} \sim A.
		\end{equation*}
			\end{Definition}

\begin{Remark}\label{Unp}
It is known that $ \bm{\xi}_{A} $ is continuous by \cite[4.8(4)]{MR0110078}, $ \dmn \bm{\xi}_{A} $ is a Borel subset of $ \Real{n} $ by \cite[3.5]{2017arXiv170309561M}, $ \bm{\xi}_{A}^{-1}\{a\} $ is a convex subset \mbox{of $ \Real{n} $} whenever $ a \in A $ by \cite[4.8(2)]{MR0110078} and 
\begin{equation}\label{Unp:1}
	\Leb{n}(\Real{n} \sim \dmn \bm{\xi}_{A}) =0
\end{equation}
by \cite[4.8(3)]{MR0110078} and Rademacher's theorem \cite[3.1.6]{MR0257325}. 
\end{Remark}

\begin{Remark}\label{psi is an homeomorphism}
Noting \ref{Unp}, we readily infer that for every $ 0 < r < \infty $ the map $ \bm{\psi}_{A}| U(A) \cap S(A,r) $ is an homeomorphism with 
\begin{equation*}
 (\bm{\psi}_{A}| U(A) \cap S(A,r))^{-1}(a,u) = a + ru \quad \textrm{whenever $ (a,u) \in \bm{\psi}_{A}[U(A) \cap S(A,r)] $.}
\end{equation*}
\end{Remark}

\begin{Remark}\label{n.p.p. along radial directions}
We notice that if $ v \in \Real{n} \sim \{0\} $, $ a \in A $ and $ |v| = \bm{\delta}_{A}(a+v) $ then 
	\begin{equation*}
		a + t v \in U(A) \quad \textrm{and} \quad   \bm{\xi}_{A}(a + t v) = a 
	\end{equation*}
	whenever $ 0 < t < 1 $.
\end{Remark}

\begin{Lemma}\label{range of the ap diff of the n.p.p.} 
Suppose $ A \subseteq \Real{n} $ is closed, $ x \in U(A) $, $ \bm{\xi}_{A} $ is approximately differentiable at $ x $ and $ T = \Real{n} \cap \{v: v \bullet \bm{\nu}_{A}(x) = 0 \}$.

Then $\bm{\delta}_{A} $ is differentiable at $x$, $ \bm{\nu}_{A} $ is approximately differentiable at $ x $, 
\begin{equation*}
ap \Der \bm{\xi}_{A}(x) \bullet \bm{\nu}_{A}(x) = 0 \quad \textrm{and} \quad \ap \Der \bm{\nu}_{A}(x) = | x - \bm{\xi}_{A}(x)|^{-1} (T_{\natural}  -  \ap \Der \bm{\xi}_{A}(x) ).
\end{equation*}
In particular $ \ker \ap \Der \bm{\psi}_{A}(x) \subseteq T^{\perp} $.
\end{Lemma}

\begin{proof}
Since $\bm{\delta}_{A}(y)= | y-\bm{\xi}_{A}(y)| $ for $ y \in \dmn \bm{\xi}_{A} $, we use \ref{composition and approx diff}, \ref{approx. versus pointwise diff. for Lip functions} and \cite[4.8(3)]{MR0110078} to deduce that $ \bm{\delta}_{A} $ is differentiable at $ x $ and 
\begin{equation}\label{range of the ap diff of the n.p.p. eq 1} 
\Der \bm{\delta}_{A}(x)(v) = \bm{\nu}_{A}(x) \bullet v \quad \textrm{for $ v \in \Real{n} $.}
\end{equation}
It follows that $\bm{\nu}_{A}$ is approximately differentiable at $x$ and employing \eqref{range of the ap diff of the n.p.p. eq 1} one computes 
\begin{flalign*}
	\ap \Der \bm{\nu}_{A}(x)(v) = \frac{T_{\natural}(v)-\ap \Der \bm{\xi}_{A}(x)(v)}{\bm{\delta}_{A}(x)} \quad \textrm{for $ v \in \Real{n} $.}
\end{flalign*}
Then we readily infer that $ \ker \ap \Der \bm{\psi}_{A}(x) \subseteq T^{\perp} $.

If $ r = |x - \bm{\xi}_{A}(x)| $ we use the continuity of $ \bm{\xi}_{A} $ at $ x $ (see \ref{Unp}) to select $ 0 < \delta < r $ such that $|\bm{\xi}_{A}(z) - \bm{\xi}_{A}(x)| \leq r$  and
\begin{equation}\label{range of the ap diff of the n.p.p. 1}
(\bm{\xi}_{A}(z) - x)\bullet \bm{\nu}_{A}(x) = (\bm{\xi}_{A}(z)-\bm{\xi}_{A}(x))\bullet \bm{\nu}_{A}(x) - r \leq 0
\end{equation}
whenever $ z \in \mathbf{U}(x,\delta) \cap \dmn \bm{\xi}_{A} $. Since $ |\bm{\xi}_{A} - x | \geq r $ and $ T_{\natural}(x - \bm{\xi}_{A}(x)) =0 $ we use \eqref{range of the ap diff of the n.p.p. 1} to infer
	\begin{equation*}
	\big(r^{2}- | T_{\natural}(\bm{\xi}_{A}(z)- \bm{\xi}_{A}(x))|^{2}\big)^{1/2} \leq |(\bm{\xi}_{A}(z) - x)\bullet \bm{\nu}_{A}(x)|=-(\bm{\xi}_{A}(z) - x)\bullet \bm{\nu}_{A}(x),
\end{equation*}
\begin{equation}\label{range of the ap diff of the n.p.p. 2}
(\bm{\xi}_{A}(z) - \bm{\xi}_{A}(x)) \bullet \bm{\nu}_{A}(x) + (r^{2} - |T_{\natural}(\bm{\xi}_{A}(z)-\bm{\xi}_{A}(x))|^{2})^{1/2} \leq r,
\end{equation}
for $ z \in \mathbf{U}(x,\delta) \cap \dmn \bm{\xi}_{A} $. Employing \ref{local maximum of approximately differentiable functions} and \ref{composition and approx diff} we obtain from \eqref{range of the ap diff of the n.p.p. 2} that
\begin{equation*}
\ap \Der \bm{\xi}_{A}(x) \bullet \bm{\nu}_{A}(x) = 0.
\end{equation*}
\end{proof}

\begin{Definition}\label{definition of rho}
	If $A$ is a closed subset of $\Real{n}$, we define
	\begin{equation*}
	\rho(A,x) = \sup \{t : \bm{\delta}_{A}(\bm{\xi}_{A}(x) + t (x-\bm{\xi}_{A}(x) ))=t \bm{\delta}_{A}(x)  \},
	\end{equation*}
	whenever $x \in U(A) $.
\end{Definition}

\begin{Remark}\label{basic properties of rho}
	We notice that if $ x \in U(A) $ then $ 1 \leq \rho(A,x) \leq \infty $ and
	\begin{equation*}
\rho(A,x) \geq \lambda  \quad \textrm{if and only if} \quad \bm{\delta}_{A}(\bm{\xi}_{A}(x) + \lambda (x-\bm{\xi}_{A}(x) ))=\lambda \bm{\delta}_{A}(x) 
	\end{equation*}
for $ \lambda \geq 1 $. It follows from \ref{Unp} that $ \rho(A,\cdot): U(A) \rightarrow \Real{} \cup \{+\infty\} $ is an upper-semicontinuous function.
\end{Remark}

\begin{Definition}\label{A lambda: def}
	If $A$ is a closed subset of $\Real{n}$ and $ \lambda \geq 1 $ we define
	\begin{equation*}
		A_{\lambda} = \{ x : \rho(A,x) \geq \lambda   \}
	\end{equation*}
	and $ D(A_{\lambda}) $ to be the set of $ x \in A_{\lambda} $ such that $\bm{\xi}_{A}|A_{\lambda}$ is approximately differentiable at $x$; see \ref{approximate differentiability for functions remark}.
\end{Definition}

\begin{Remark}\label{A_lambda and positive reach}
	If $ 0 < R = \reach(A)$, $ 0 < r < R$ and $ 0 < \bm{\delta}_{A}(x)\leq r $ it follows from \cite[4.8(6)]{MR0110078} that
	\begin{equation*}
	\sup \{ t : \bm{\xi}_{A}(\bm{\xi}_{A}(x)  + t (x-\bm{\xi}_{A}(x)   ))= \bm{\xi}_{A}(x)   \} \geq R/r;
	\end{equation*}
	in particular, $ \{ x : 0 < \bm{\delta}_{A}(x)\leq r  \} \subseteq A_{R/r} $.
\end{Remark}

Here we provide a thorough description of the nearest point projection $ \bm{\xi}_{A} $ on the super level sets $ A_{\lambda} $.

\begin{Lemma}\label{A lambda}
	Suppose $ A $ is a closed subset of $ \Real{n} $ and define the maps\footnote{In case $ A $ is convex, the map $ h_{t} $ is called ``dilation with center $A$'' in \cite[\S 3]{MR0417984}.} $ h_{t}$ on $ U(A) $ corresponding to $ 0 < t < \infty $ by
	\begin{equation}\label{A lambda:3}
	h_{t}(z) = \bm{\xi}_{A}(z) + t ( z - \bm{\xi}_{A}(z)) \quad \textrm{for $ z \in U(A) $.}
	\end{equation}
	
	Then the following statements hold for $ 1 < \lambda < \infty $ and $ 0 < t < \lambda $.
	\begin{enumerate}
		\item \label{lipschitzianity nearest point projection} $ \Lip(\bm{\xi}_{A}|A_{\lambda}) \leq \lambda (\lambda -1)^{-1} $ and $h_{t}|A_{\lambda}$ is a bi-Lipschitzian homeomorphism onto $ A_{\lambda/t} $ with $(h_{t}|A_{\lambda})^{-1} = h_{t^{-1}}|A_{\lambda/t} $.
		\item \label{A lambda:2} $\Leb{n}(A_{\lambda} \sim D(A_{\lambda})) =0 $.
		\item \label{differentiable extension of psi} The map $\bm{\psi}_{A}|A_{\lambda}$ has an extension $\Psi : \Real{n} \rightarrow \Real{n}\times \Real{n} $ such that $\Psi$ is differentiable at every $x\in D(A_{\lambda})$ with $\Der \Psi(x) = \ap \Der \bm{\psi}_{A}(x)$. Moreover $\ker \ap \Der \bm{\psi}_{A}(x) = \{ s\bm{\nu}_{A}(x): s \in \Real{}   \}$ whenever $x \in D(A_{\lambda})$.
		\item \label{A lambda:4} $ h_{t}[D(A_{\lambda})] \subseteq D(A_{\lambda/t}) $.
		\item \label{n.p.p.and dilations} If $ x \in D(A_{\lambda}) $ then $ h_{t^{-1}} $ is approximately differentiable at $ h_{t}(x) $ with 
		\begin{equation*}
\ap \Der h_{t^{-1}}(h_{t}(x)) =\ap \Der h_{t}(x)^{-1}, 
		\end{equation*}
		\begin{equation*}
			 \ap \Der \bm{\psi}_{A}(x) = \ap \Der \bm{\psi}_{A}(h_{t}(x)) \circ \ap \Der h_{t}(x).
		\end{equation*}
		\item \label{symmetry} If $ x \in D(A_{\lambda}) $ then the eigenvalues of $\ap \Der \bm{\xi}_{A}(x) $ and $\ap \Der \bm{\nu}_{A}(x) $ belong to the intervals $0 \leq s \leq \lambda (\lambda -1)^{-1} $ and $(1-\lambda)^{-1}\bm{\delta}_{A}(x)^{-1} \leq s \leq \bm{\delta}_{A}(x)^{-1} $, respectively. In case $\ap \Der \bm{\xi}_{A}(x)$ is a symmetric endomorphism, so are $\ap \Der \bm{\xi}_{A}(h_{t}(x))$ and $\ap \Der \bm{\nu}_{A}(h_{t}(x))$.
	\end{enumerate}
\end{Lemma}

\begin{proof*}[Proof of \eqref{lipschitzianity nearest point projection}]
If $ x \in A_{\lambda} $ and $ y \in A_{\lambda} $, then we apply \cite[4.7(1)]{2017arXiv170309561M} with $ q $, $ a $, $ b $ and $ v $ replaced by $ \lambda |x - \bm{\xi}_{A}(x)| $, $ \bm{\xi}_{A}(x)$, $ \bm{\xi}_{A}(y) $ and $ x - \bm{\xi}_{A}(x) $ respectively, to infer that 
	\begin{equation*}
	(\bm{\xi}_{A}(y)-\bm{\xi}_{A}(x)) \bullet (x - \bm{\xi}_{A}(x)) \leq (2\lambda)^{-1} |\bm{\xi}_{A}(x) - \bm{\xi}_{A}(y)|^{2},
	\end{equation*}
	and symmetrically,
	\begin{equation*}
	(\bm{\xi}_{A}(x)-\bm{\xi}_{A}(y)) \bullet (y - \bm{\xi}_{A}(y)) \leq (2\lambda)^{-1} |\bm{\xi}_{A}(x) - \bm{\xi}_{A}(y)|^{2}.
	\end{equation*}
	Combining the two equations we get
	\begin{equation*}
	|\bm{\xi}_{A}(x) - \bm{\xi}_{A}(y)||x-y| \geq (\bm{\xi}_{A}(x) - \bm{\xi}_{A}(y)) \bullet (x-y) \geq \lambda^{-1}(\lambda -1)|\bm{\xi}_{A}(x) - \bm{\xi}_{A}(y)|^{2}.
	\end{equation*}	
	
	By \ref{n.p.p. along radial directions} one infers $ \bm{\xi}_{A}(h_{t}(x)) = \bm{\xi}_{A}(x) $ and $ h_{t^{-1}}(h_{t}(x)) = x $ whenever $ x \in A_{\lambda} $, and $ h_{t}[A_{\lambda}] \subseteq A_{\lambda / t} $. Since $ 0 < t^{-1} < \lambda / t $, the same conclusions hold with $ \lambda $ and $ t $ replaced by $ \lambda / t $ and $ t^{-1} $ respectively. Henceforth \eqref{lipschitzianity nearest point projection} is proved. 
\end{proof*}
\begin{proof*}[Proof of \eqref{A lambda:2}]
Since $ \bm{\xi}_{A}|A_{\lambda} $ is Lipschitzian then $ \Leb{n}(A_{\lambda} \sim D(A_{\lambda}) ) =0 $ by \cite[2.11]{2017arXiv170107286S}.
\end{proof*}
\begin{proof*}[Proof of \eqref{differentiable extension of psi}]
Since $ \bm{\xi}_{A}|A_{\lambda}$ is Lipschitzian there exists a Lipschitzian function $F: \Real{n} \rightarrow \Real{n}$ such that $ F|A_{\lambda} = \bm{\xi}_{A}|A_{\lambda} $ by \cite[2.10.43]{MR0257325}. Then, by \ref{approx. versus pointwise diff. for Lip functions}, the map $F$ is differentiable at every $x\in D(A_{\lambda})$ with
		\begin{equation*}
			\Der F(x) = \ap \Der \bm{\xi}_{A}(x).
		\end{equation*}
	If $ x \in D(A_{\lambda})$ then $ x + s \bm{\nu}_{A}(x) \in A_{\lambda} $ and
	\begin{equation*}
		 F(x+s\bm{\nu}_{A}(x)) = \bm{\xi}_{A}(x+ s\bm{\nu}_{A}(x)) = \bm{\xi}_{A}(x)
	\end{equation*}
	for $ -\bm{\delta}_{A}(x) < s < (\lambda -1)\bm{\delta}_{A}(x) $. Differentiating with respect to $ s $ we get that
	\begin{equation*}
	\ap\Der\bm{\xi}_{A}(x)(\bm{\nu}_{A}(x)) = \Der F(x)(\bm{\nu}_{A}(x)) =0
	\end{equation*}
	and $ \ap\Der\bm{\nu}_{A}(x)(\bm{\nu}_{A}(x)) =0 $ by \ref{range of the ap diff of the n.p.p.}. Let $ G : \Real{n} \rightarrow \Real{n} $ be any function such that $ G(x) = \bm{\delta}_{A}(x)^{-1}(x-F(x) ) $ for $ x \in \Real{n} \sim A $. Noting \ref{range of the ap diff of the n.p.p.} and \cite[2.8]{2017arXiv170107286S} we infer that $ G $ is differentiable at every $ x \in D(A_{\lambda}) $ with $ \Der G(x) = \ap \Der \bm{\nu}_{A}(x) $. Henceforth $ \Psi = (F,G) $ and \eqref{differentiable extension of psi} is proved.
\end{proof*}
\begin{proof*}[Proof of \eqref{A lambda:4} and \eqref{n.p.p.and dilations}]
	Let $ x \in D(A_{\lambda}) $ and $ y = h_{t}(x) $. Then $ h_{t} $ is approximately differentiable at $ x $ and, noting \eqref{lipschitzianity nearest point projection}, we can use \ref{bi-Lip and inverse of ap Df} and \cite[Theorem 1]{MR1100645} to infer that $ \ap \Der h_{t}(x) $ is an isomorphism of $\Real{n}$ and 
	\begin{equation*}
	\Ldensity{n}{\Real{n} \sim  A_{\lambda/t}}{y} = 0.
	\end{equation*}
	For $ \epsilon > 0 $ we define
	\begin{equation*}
	P_{\epsilon} = A_{\lambda} \cap \{ w : |h_{t}(w) - h_{t}(x) - \ap \Der h_{t}(x)(w-x) | \geq \epsilon |w-x| \},
	\end{equation*}
	\begin{equation*}
Q_{\epsilon}= A_{\lambda/t}\cap \{ z :  |h_{t^{-1}}(z) - x -  \ap \Der h_{t}(x)^{-1}(z-y) | \geq \epsilon |z-y| \},
	\end{equation*}
we observe that $ Q_{\epsilon} \subseteq h_{t}(P_{C\epsilon}) $ for $ C= \| \ap \Der h_{t}(x)^{-1}\|^{-1}(\Lip(h_{t}|A_{\lambda})^{-1})^{-1} $ and 
\begin{equation*}
	\mathbf{B}(h_{t}(x),r) \cap Q_{\epsilon} \subseteq h_{t}[ P_{C\epsilon}  \cap \mathbf{B}(x, (\Lip(h_{t}|A_{\lambda})^{-1})r)  ] \quad \textrm{for $ r > 0 $,}
\end{equation*}
whence we deduce that
\begin{equation*}
	\Ldensity{n}{Q_{\epsilon}}{h_{t}(x)} =0 \quad \textrm{for every $ \epsilon > 0 $,}
\end{equation*}
the map $ h_{t^{-1}} $ is approximately differentiable at $ y $ and
	\begin{equation*}
\ap \Der h_{t^{-1}}(y)  = \ap \Der h_{t}(x)^{-1}. 
\end{equation*}
	Let $\Psi$ be an extension of $ \bm{\psi}_{A}|A_{\lambda} $ given by \eqref{differentiable extension of psi}. If $ z \in A_{\lambda/t} $, being $ \lambda > 1 $ and noting \ref{n.p.p. along radial directions}, we get that
	\begin{equation*}
	\Psi(h_{t^{-1}}(z)) = \bm{\psi}_{A}(h_{t^{-1}}(z)) = \bm{\psi}_{A}(z)
	\end{equation*}
and we use \ref{composition and approx diff} to infer that $ \bm{\psi}_{A} $ is approximately differentiable at $ y $ with 
	\begin{equation*}
	\ap \Der \bm{\psi}_{A}(y) = \ap \Der \bm{\psi}_{A}(x) \circ \ap \Der h_{t^{-1}}( y).
	\end{equation*}
\end{proof*}
\begin{proof*}[Proof of \eqref{symmetry}]
	If $ \mu \in \Real{} $, $ v \in \mathbf{S}^{n-1}$ and $\ap \Der \bm{\xi}_{A}(x)(v) = \mu v $ then, noting that $\ap \Der h_{s}(x) $ is injective for $ 0 < s < \lambda $ by \eqref{n.p.p.and dilations}, we infer that
	\begin{equation*}
		 (1-s)\mu + s \neq 0 \quad \textrm{for $ 0 < s < \lambda $,}
	\end{equation*}
	whence we deduce that
	\begin{equation*}
	0 \leq \mu \leq \lambda (\lambda -1)^{-1}.	
	\end{equation*}
If $ \mu \neq 0 $, $ v \in \mathbf{S}^{n-1}$ and $\ap \Der \bm{\nu}_{A}(x)(v) = \mu v$ then 
\begin{equation*}
	v \bullet \bm{\nu}_{A}(x)=0 \quad \textrm{and} \quad  \ap \Der \bm{\xi}_{A}(x)(v) = (1 - \bm{\delta}_{A}(x)\mu)v
\end{equation*}
by \ref{range of the ap diff of the n.p.p.}, which implies $ (1-\lambda)^{-1}\bm{\delta}_{A}(x)^{-1}\leq \mu \leq \bm{\delta}_{A}(x)^{-1} $.
	
	If $\ap \Der \bm{\xi}_{A}(x)$ is symmetric, then there exists an orthonormal basis $ v_{1}, \ldots , v_{n} $ of $\Real{n} $ and $ 0 \leq \mu_{1} \leq  \ldots \leq \mu_{n} $ such that $ \ap \Der \bm{\xi}_{A}(x)(v_{i}) = \mu_{i}v_{i} $ for $ i = 1 , \ldots , n $ and \eqref{n.p.p.and dilations} implies that
	\begin{equation*}
		\ap \Der \bm{\xi}_{A}(h_{t}(x))(v_{i}) = \mu_{i} ((1-t)\mu_{i} + t)^{-1} v_{i} \quad \textrm{whenever $ i = 1 , \ldots , n $}.
	\end{equation*}
Therefore $ \ap \Der \bm{\xi}_{A}(h_{t}(x)) $ is symmetric and so is  $ \ap \Der \bm{\nu}_{A}(h_{t}(x)) $ by \ref{range of the ap diff of the n.p.p.}.
\end{proof*}

\begin{Remark}\label{range of the ap diff of the n.p.p. and dilations}
	Combining \ref{range of the ap diff of the n.p.p.} and \ref{A lambda}\eqref{n.p.p.and dilations}, if $ 1 < \lambda < \infty $, $ 0 < t < \lambda $, $ x \in D(A_{\lambda}) $ and $ T = \Real{n} \cap \{v : v \bullet \bm{\nu}_{A}(x) =0 \} $, then 
	\begin{equation*}
	\im \ap \Der \bm{\xi}_{A}(h_{t}(x)) = \im \ap \Der \bm{\xi}_{A}(x) \subseteq T,
	\end{equation*}
	\begin{equation*}
	\im \ap \Der \bm{\nu}_{A}(h_{t}(x)) = \im \ap \Der \bm{\nu}_{A}(x) \subseteq T.
	\end{equation*}  
\end{Remark}

Here the tangential and curvature properties of the distance sets $ S(A,r) $ are expressed in terms of the spherical image map of $ A $ and its approximate differential.

\begin{Lemma}\label{fine properties of level sets}
	If $ A $ is a closed subset of $ \Real{n} $ then for $ \Leb{1} $ a.e.\ $ r > 0 $ and for $ \Haus{n-1} $ a.e.\ $ x \in S(A,r) $ the following three statements hold:
\begin{equation*}
\Haus{n-1}\bigg( S(A,r) \sim \bigcup_{\lambda > 1}D(A_{\lambda})\bigg) =0,
\end{equation*}
	\begin{equation*}
	\ap \Tan(S(A,r),x) = \{ v : v \bullet \bm{\nu}_{A}(x) =0  \},
	\end{equation*}
	\begin{equation*}
	\ap \Der^{2} S(A,r)(x)(u,v) \bullet \bm{\nu}_{A}(x) = - \ap \Der \bm{\nu}_{A}(x)(u) \bullet v
	\end{equation*}
	for $ u,v \in \ap \Tan(S(A,r),x) $. 
\end{Lemma}	

\begin{proof}
	We define $P_{r} = \{  x : \bm{\delta}_{A}(x) \leq r \}$ for $ r > 0 $ and $B = \bigcup_{\lambda > 1}A_{\lambda}$.
First we prove that 
\begin{equation*}
	S(A,r) \cap B = \partial^{+} P_{r}  \quad \textrm{for every $ r > 0 $.}
\end{equation*}
Let $ x \in \partial^{+}P_{r} $. Then $ x \in S(A,r) $ and we choose $ a \in A $ with $ |x-a|= r $, $ u \in \mathbf{S}^{n-1} $ and $ s > 0 $ such that $ \mathbf{U}(x+su,s) \cap P_{r} = \varnothing $. Noting that $\bm{\delta}_{A}(x+su) > r$ we apply \cite[4.9]{MR0110078} to infer that 
\begin{equation*}
 \quad 	s = \bm{\delta}_{P_{r}}(x+su) = \bm{\delta}_{A}(x+su)-r
\end{equation*}
whence we deduce that $r+s \leq |x+su-a| $ and $r \leq u \bullet (x-a)$. It follows that $ x-a $ and $ u $ must be linearly dependent and $ x-a = ru $. Noting \ref{n.p.p. along radial directions} we conclude that $ \rho(A,x) \geq r^{-1}(r+s) $. We assume now $ x \in A_{\lambda} \cap S(A,r) $ for $ \lambda > 1 $. Since $ \bm{\delta}_{A}(\bm{\xi}_{A}(x) + \lambda (x-\bm{\xi}_{A}(x) )  ) = \lambda r $ it follows from \cite[4.9]{MR0110078} that 
\begin{equation*}
\bm{\delta}_{P_{r}}(\bm{\xi}_{A}(x) + \lambda (x-\bm{\xi}_{A}(x) )) = (\lambda-1 )r 
\end{equation*}
and, noting that $\bm{\xi}_{A}(x) + \lambda (x-\bm{\xi}_{A}(x) ) = x + (\lambda-1)r \bm{\nu}_{A}(x) $, we conclude that $ x \in \partial^{+}P_{r} $.

It follows from \ref{Extension of a lemme of Rataj-Winter} that $\Haus{n-1}(S(A,r) \sim B ) =0$ for all, but countably many $ r > 0 $, whence we deduce using \ref{A lambda}\eqref{A lambda:2} and Coarea formula that
\begin{equation}\label{approximate diff nearest point pr:1}
	\Haus{n-1}\bigg( S(A,r) \sim \bigcup_{\lambda > 1}D(A_{\lambda})\bigg) =0 \quad \textrm{for $ \Leb{1} $ a.e.\ $ r > 0 $.}
\end{equation}

 It follows from \ref{A lambda}\eqref{differentiable extension of psi} and \cite[2.10.19(4), 3.2.16]{MR0257325} that for all $ r > 0 $, $ \lambda > 1 $ and for $ \Haus{n-1} $ a.e.\ $ x \in S(A,r) \cap D(A_{\lambda}) $,
 \begin{equation}\label{approximate diff nearest point pr:4}
 	\Hdensity{n-1}{S(A,r) \sim A_{\lambda}}{x} =0
 \end{equation}
and $ \bm{\psi}_{A} $ is $(\Haus{n-1}\restrict S(A,r), n-1) $ approximately differentiable\footnote{Given a measure $ \phi $ on a normed vector space and a positive integer $ m $, we refer to \cite[3.2.16]{MR0257325} for the notion of $(\mu, m)$ approximate differentiability.} at $ x $ with
\begin{equation}\label{approximate diff nearest point pr:2}
(\Haus{n-1}\restrict S(A,r), n-1)\ap \Der \bm{\psi}_{A}(x) = \ap \Der \bm{\psi}_{A}(x).
\end{equation}

Moreover we claim that for $ \Leb{1} $ a.e.\ $ r > 0 $ and for $ \Haus{n-1} $ a.e.\ $ x \in S(A,r) $
\begin{equation}\label{approximate diff nearest point pr:3}
	\ap \Tan(S(A,r),x) = \{ v : v \bullet \bm{\nu}_{A}(x) =0  \}.
\end{equation}
To prove \eqref{approximate diff nearest point pr:3}, first we notice that it follows from \cite[3.1.6, 3.2.11, 3.1.21]{MR0257325}, \cite[4.8(3)]{MR0110078} and \eqref{Unp:1} that $\bm{\delta}_{A}$ is differentiable at $ x $ with $  \grad \bm{\delta}_{A}(x) = \bm{\nu}_{A}(x) $ and $ \Tan(S(A,r), x) \subseteq \{ v : v \bullet \grad \bm{\delta}_{A}(x) =0   \} $  for $ \Leb{1} $ a.e.\ $ r > 0 $ and for $ \Haus{n-1} $ a.e.\ $ x \in S(A,r) $; second we employ \ref{structure of level sets} and \cite[3.23]{2017arXiv170107286S}.

Combining \eqref{approximate diff nearest point pr:1}-\eqref{approximate diff nearest point pr:3} with \ref{A lambda}\eqref{lipschitzianity nearest point projection} and \cite[3.25]{2017arXiv170107286S} we conclude that
\begin{equation*}
\ap \Der^{2} S(A,r)(x)(u,v) \bullet \bm{\nu}_{A}(x) = - \ap \Der \bm{\nu}_{A}(x)(u) \bullet v 
\end{equation*}
for $ u,v \in \ap \Tan(S(A,r),x) $, for $ \Haus{n-1} $ a.e.\ $ x \in S(A,r) $ and \mbox{for $ \Leb{1} $ a.e.\ $ r > 0 $.}
\end{proof}

\begin{Definition}\label{regular points}
	If $ A \subseteq \Real{n} $ is a closed set we say that $ x \in U(A) $ is a \emph{regular point of $ \bm{\xi}_{A} $} if and only if $ \ap \lim_{y \to x} \rho(A,y) = \rho(A,x)>1$ and $ \bm{\xi}_{A}$ is approximately differentiable at $ x $ with symmetric approximate differential.
	
	The set of regular points of $ \bm{\xi}_{A} $ is denoted by $ R(A)$.
\end{Definition}

\begin{Theorem}\label{approximate diff of nearest point pr}
	If $ A $ is a closed subset of $ \Real{n} $ then $ \Leb{n}(\Real{n} \sim (R(A) \cup A)) =0 $. If $ x \in R(A) $ then $ \bm{\xi}_{A}(x) + t(x - \bm{\xi}_{A}(x)) \in R(A)$ for every $ 0 < t < \rho(A,x) $.
\end{Theorem}

\begin{proof}
One infers from \ref{fine properties of level sets}, \ref{range of the ap diff of the n.p.p.} and Coarea formula that $\rho(A,x) > 1$ and $\bm{\xi}_{A}$ is approximately differentiable with symmetric approximate differential for $ \Leb{n} $ a.e.\ $ x \in \Real{n} \sim A $. Since $ \rho(A, \cdot) $ is a Borel function by \ref{basic properties of rho}, it follows from \cite[2.9.13]{MR0257325} that $\ap \lim_{y \to x} \rho(A,y) = \rho(A,x)$ for $ \Leb{n} $ a.e.\ $ x \in U(A) $. Therefore,
\begin{equation*}
	\Leb{n}(\Real{n} \sim (R(A) \cup A)) =0.
\end{equation*} 

If $ x \in R(A) $ and $ 0 < t < \rho(A,x) $ we choose $ \lambda $ such that $ t < \lambda < \rho(A,x) $ and $ \lambda > 1 $ and we notice that $ x \in D(A_{\lambda}) $. It follows from \ref{A lambda}\eqref{A lambda:4}\eqref{symmetry} that $ \bm{\xi}_{A} $ is approximately differentiable at $h_{t}(x)$ (see \eqref{A lambda:3}) with symmetric approximate differential, 
\begin{equation*}
\Ldensity{n}{\Real{n} \sim A_{\lambda/t}}{h_{t}(x)} =0, \qquad \ap\liminf_{y \to h_{t}(x)}\rho(A,y) \geq \lambda /t.
\end{equation*}
Since $ \rho(A,h_{t}(x)) = t^{-1}\rho(A,x) $, we conclude that
\begin{equation*}
\ap\liminf_{y \to h_{t}(x)}\rho(A,y) \geq \rho(A,h_{t}(x)) > 1
\end{equation*}
and it follows from \ref{basic properties of rho} that $ h_{t}(x) \in R(A) $.
\end{proof}

\begin{Remark}
The fact that $ \bm{\xi}_{A} $ is approximately differentiable with symmetric approximate differential at $ \Leb{n} $ a.e.\ $ x \in U(A) $ can be alternatively deduced from \cite{MR0310150}.
\end{Remark}

\begin{Remark}\label{approximate diff of nearest point pr: remark}
	It follows from Coarea formula and \ref{approximate diff of nearest point pr} that
	\begin{equation*}
		\Haus{n-1}(S(A,r) \sim R(A)) =0 \quad \textrm{for $ \Leb{1} $ a.e.\ $ r > 0 $.}
	\end{equation*}
\end{Remark}

\begin{Definition}\label{second order differentiable level sets: definition} If $ A \subseteq \Real{n} $ is a closed set, $ 1 < \lambda < \infty $ and $ 0 < r < \infty $ then we define
	\begin{equation*}
		S_{\lambda}(A,r) = S(A,r) \cap A_{\lambda}, 
	\end{equation*}
\end{Definition}

\begin{Remark}\label{remark on S lambda}
	If $ r > 0 $ we can readily check the following properties.
	\begin{enumerate}
	\item $\bm{\psi}_{A}|S_{\lambda}(A,r)$ is a bi-Lipschitzian homeomorphism by \ref{psi is an homeomorphism} and \ref{A lambda}\eqref{lipschitzianity nearest point projection}.
	\item $\bm{\psi}_{A}[S_{\lambda}(A,r)] = (A \times \mathbf{S}^{n-1}) \cap \{ (a,u) : \bm{\delta}_{A}(a + \lambda r u ) = \lambda r  \} $ (using \ref{n.p.p. along radial directions} and \ref{basic properties of rho}), whence we deduce that $ \bm{\psi}_{A}[S_{\lambda}(A,r)] $ is a closed subset of $A \times \mathbf{S}^{n-1} $ and 
	\begin{equation*}
\bm{\psi}_{A}[S_{\lambda}(A,r)] \subseteq \bm{\psi}_{A}[S_{\lambda}(A,s)] \quad \textrm{if $ 0 < s < r < \infty $.}
	\end{equation*}
	\item It follows from \ref{structure of level sets}\eqref{structure of level sets:1} that $\bm{\psi}_{A}[S_{\lambda}(A,r)] |K$ is $ n-1 $ rectifiable for every $ K \subseteq \Real{n} $ compact.
	\item  If $\reach(A) = R > 0$ and $ 0 < r < R $ it follows from \ref{A_lambda and positive reach} that
	\begin{equation*}
	S(A,r)= S_{R/r}(A,r).
	\end{equation*}
	\end{enumerate}

	\end{Remark}

\section{Second fundamental form}\label{section: s.f.f.}

In this section we introduce the second fundamental form in \eqref{intro: sff} and we prove theorem \ref{relating principal curvatures}.

\begin{Definition}\label{generalized unt normal bundle}
	Suppose $ A $ is a closed subset of $ \Real{n} $. We define 
	\begin{equation*}
	N(A) = (A \times \mathbf{S}^{n-1}) \cap \{ (a,u) :  \bm{\delta}_{A}(a+su)=s \; \textrm{for some $ s > 0 $}\}.
	\end{equation*}
	Moreover we let $ N(A,a) = \{ v : (a,v) \in N(A)   \} $ for $ a \in A $.
\end{Definition}

\begin{Remark}\label{Comparison with HLW04 and MS17}
	 We notice that $ N(A) $ coincides with the \textit{normal bundle of $A$} introduced in \cite[\S 2.1]{MR2031455} and $ N(A) \subseteq \Nor(A) $, see \cite[4.4]{MR0110078} or \cite[3.1.21]{MR0257325}. If $ \reach A > 0 $ then $N(A,a) = \Nor(A,a) \cap \mathbf{S}^{n-1} $ for $ a \in A $ by \cite[4.8(12)]{MR0110078}.
\end{Remark}

\begin{Remark}\label{bi-lipschitz parametrization of the unit normal bundle}
		If $ 1 < \lambda < \infty $, $(a,u) \in A \times \mathbf{S}^{n-1} $ and $ \bm{\delta}_{A}(a+su) = s $ for some $ s > 0 $ it follows from \ref{n.p.p. along radial directions} that $ a + \lambda^{-1}su \in A_{\lambda} $ and $ \bm{\psi}_{A}(a+\lambda^{-1}su)=(a,u) $. Then we readily infer that
	\begin{equation*}
		N(A) = \bm{\psi}_{A}[A_{\lambda}] = \bigcup_{r >0}\bm{\psi}_{A}[S_{\lambda}(A,r)].
	\end{equation*}
It follows from \ref{remark on S lambda} that \textit{$ N(A) $ is a countably $ n-1 $ rectifiable Borel subset of $ \Real{n} \times \mathbf{S}^{n-1} $.} This fact has been already noticed in \cite[p.\ 243]{MR2031455}.
\end{Remark}

\begin{Definition}\label{regular points of N(A)}
	If $ x \in R(A) $ then \emph{$ \bm{\psi}_{A}(x) $ is a regular point of $ N(A) $.} We denote the set of all regular points of $N(A)$ by $R(N(A))$.
\end{Definition}

\begin{Remark}\label{remark on regular points}
	It follows from \ref{bi-lipschitz parametrization of the unit normal bundle}, \ref{remark on S lambda} and \ref{approximate diff of nearest point pr: remark} that
	\begin{equation*}
	\Haus{n-1}(N(A)\sim R(N(A)) =0.
	\end{equation*}
	
	Moreover it follows from \ref{approximate diff of nearest point pr} that if $ (a,u)\in R(N(A)) $ then $ a + ru \in R(A) $ for $ 0 < r < \sup\{ s : \bm{\delta}_{A}(a+su) =s  \} $.
\end{Remark}

The following lemma ensures that the definition in \ref{second fundamental form} is well posed.

\begin{Lemma}\label{preparation for definition s.f.f.}
	Suppose $ A \subseteq \Real{n} $ is a closed set, $x \in R(A)$, $ 0 < t < \rho(A,x) $ and $y = \bm{\xi}_{A}(x) + t(x - \bm{\xi}_{A}(x)) $, then the following two statements hold.
	
	\begin{enumerate}
		\item \label{preparation for definition s.f.f.: 1} If $ v,v_{1}, v_{2} \in \Real{n} $ are such that $ \ap \Der \bm{\xi}_{A}(x)(v_{1}) = \ap \Der \bm{\xi}_{A}(x)(v_{2}) $, then
		\begin{flalign*}
			& \ap \Der \bm{\xi}_{A}(x)(v) \bullet \ap \Der \bm{\nu}_{A}(x)(v_{1}) = \ap \Der \bm{\xi}_{A}(x)(v) \bullet \ap \Der \bm{\nu}_{A}(x)(v_{2}),\\
			& \ap \Der \bm{\xi}_{A}(x)(v_{1}) \bullet \ap \Der \bm{\nu}_{A}(x)(v) = \ap \Der \bm{\xi}_{A}(x)(v) \bullet \ap \Der \bm{\nu}_{A}(x)(v_{1}).
		\end{flalign*}
		\item \label{preparation for definition s.f.f.:2} If $ v,w,v_{1}, w_{1} \in \Real{n} $ are such that $ \ap \Der \bm{\xi}_{A}(y)(w) = \ap \Der \bm{\xi}_{A}(x)(v) $ and $ \ap \Der \bm{\xi}_{A}(y)(w_{1}) = \ap \Der \bm{\xi}_{A}(x)(v_{1})$, then 
		\begin{equation*}
			 \ap \Der \bm{\nu}_{A}(x)(v_{1}) \bullet \ap \Der \bm{\xi}_{A}(x)(v) = \ap \Der \bm{\nu}_{A}(y)(w_{1}) \bullet \ap \Der \bm{\xi}_{A}(y)(w).
		\end{equation*}
	\end{enumerate}
\end{Lemma}

\begin{proof}
Let $ r = |x - \bm{\xi}_{A}(x)| $ and we recall that $ x \in D(A_{\lambda}) $ for $ 1 < \lambda < \rho(A,x) $. To prove \eqref{preparation for definition s.f.f.: 1} we compute, using \ref{range of the ap diff of the n.p.p.} and \ref{A lambda}\eqref{differentiable extension of psi},
	\begin{flalign*}
		&	\ap \Der \bm{\xi}_{A}(x)(v) \bullet \ap \Der \bm{\nu}_{A}(x)(v_{1}) \\
		& \quad = r^{-1} v \bullet [ \ap \Der \bm{\xi}_{A}(x)(v_{1}) - ( \ap \Der \bm{\xi}_{A}(x)\circ \ap \Der \bm{\xi}_{A}(x) )(v_{1})  ] \\
		& \quad = r^{-1} v \bullet [ \ap \Der \bm{\xi}_{A}(x)(v_{2}) - ( \ap \Der \bm{\xi}_{A}(x)\circ \ap \Der \bm{\xi}_{A}(x) )(v_{2})  ] \\
		&\quad = \ap \Der \bm{\xi}_{A}(x)(v) \bullet \ap \Der \bm{\nu}_{A}(x)(v_{2}),\\
		& \ap \Der \bm{\xi}_{A}(x)(v) \bullet \ap \Der \bm{\nu}_{A}(x)(v_{1}) \\
		& \quad = r^{-1} v \bullet [ \ap \Der \bm{\xi}_{A}(x)(v_{1}) - ( \ap \Der \bm{\xi}_{A}(x)\circ \ap \Der \bm{\xi}_{A}(x))(v_{1})  ] \\
		& \quad = r^{-1} \ap \Der \bm{\xi}_{A}(x)(v_{1}) \bullet [ v - \ap \Der \bm{\xi}_{A}(x)(v)    ] \\
		& \quad = \ap \Der \bm{\xi}_{A}(x)(v_{1}) \bullet \ap \Der \bm{\nu}_{A}(x)(v);
	\end{flalign*}
	to prove \eqref{preparation for definition s.f.f.:2} we compute, using \ref{range of the ap diff of the n.p.p.} and \ref{A lambda}\eqref{differentiable extension of psi}\eqref{n.p.p.and dilations}\eqref{symmetry},
	\begin{flalign*}
		& \ap \Der \bm{\xi}_{A}(y)(w_{1}) = \ap \Der \bm{\xi}_{A}(x)(v_{1}) = \ap \Der \bm{\xi}_{A}(x)(T_{\natural}(v_{1})) \\
		& \qquad = \ap \Der \bm{\xi}_{A}(y)[  \ap \Der \bm{\xi}_{A}(x)(v_{1}) + t( T_{\natural}(v_{1}) - \ap \Der \bm{\xi}_{A}(x)(v_{1}) )    ] \\
		& \qquad = \ap \Der \bm{\xi}_{A}(y)[  \ap \Der \bm{\xi}_{A}(y)(w_{1}) + tr \ap \Der \bm{\nu}_{A}(x)(v_{1})    ], \\
		& t^{-1}r^{-1}[ \ap \Der \bm{\xi}_{A}(y)(w_{1}) - (  \ap \Der \bm{\xi}_{A}(y) \circ \ap \Der \bm{\xi}_{A}(y)   )(w_{1})    ] \\
		& \qquad = (\ap \Der \bm{\xi}_{A}(y) \circ \ap \Der \bm{\nu}_{A}(x))(v_{1}),\\
		& \ap \Der \bm{\nu}_{A}(x)(v_{1}) \bullet \ap \Der \bm{\xi}_{A}(x)(v) \\
		& \quad = \ap \Der \bm{\nu}_{A}(x)(v_{1}) \bullet \ap \Der \bm{\xi}_{A}(y)(w) \\
		& \quad = ( \ap \Der \bm{\xi}_{A}(y) \circ \ap \Der \bm{\nu}_{A}(x)    )(v_{1}) \bullet w \\
		& \quad = t^{-1}r^{-1}[ \ap \Der \bm{\xi}_{A}(y)(w_{1}) - (  \ap \Der \bm{\xi}_{A}(y) \circ \ap \Der \bm{\xi}_{A}(y)   )(w_{1}) ] \bullet w \\
		& \quad = \ap \Der \bm{\nu}_{A}(y) (w_{1}) \bullet \ap \Der \bm{\xi}_{A}(y)(w).
	\end{flalign*}
\end{proof}

	\begin{Definition}\label{second fundamental form}
		Suppose $A$ is a closed subset of $ \Real{n} $ and $ (a,u) \in R(N(A)) $. 
		
		We define
			\begin{equation*}
			T_{A}(a,u) = \im \ap \Der \bm{\xi}_{A}(x) \quad \textrm{and} \quad Q_{A}(a,u)(\tau,\tau_{1}) = \tau \bullet \ap \Der \bm{\nu}_{A}(x)(v_{1}),
			\end{equation*}
			whenever $ x $ is a regular point of $ \bm{\xi}_{A}$ such that $ \bm{\psi}_{A}(x)= (a,u) $, $ \tau \in T_{A}(a,u) $, $  \tau_{1} \in T_{A}(a,u) $ and $ v_{1} \in \Real{n} $ such that $ \ap \Der \bm{\xi}_{A}(x)(v_{1}) = \tau_{1}$.
			
			We call $ Q_{A}(a,u) $ \textit{second fundamental form of $ A $ at $ a $ in the direction $ u $.}
	\end{Definition}

\begin{Lemma}\label{perpendicularity of the normal bundle and estimate s.f.f.}
If $ A \subseteq \Real{n} $ is a closed set and $(a,u)\in R(N(A))$ then
\begin{equation*}
Q_{A}(a,u) : T_{A}(a,u) \times T_{A}(a,u)  \rightarrow \Real{}
\end{equation*}
is a symmetric bilinear form and $T_{A}(a,u) \subseteq \{v : v \bullet u =0\} $. Moreover if $ r > 0 $ and \mbox{$ \bm{\delta}_{A}(a+ru)=r$,} then 
\begin{equation*}
\quad Q_{A}(a,u)(\tau,\tau)\geq - r^{-1}|\tau|^{2} \quad \textrm{whenever $ \tau \in T_{A}(a,u) $}.
\end{equation*}	
\end{Lemma}

\begin{proof}
If $x$ and $y$ are regular points of $\bm{\xi}_{A} $ such that $\bm{\psi}_{A}(x)= (a,u) = \bm{\psi}_{A}(y) $ then $	y =\bm{\xi}_{A}(x) + (\bm{\delta}_{A}(y)/\bm{\delta}_{A}(x))(x - \bm{\xi}_{A}(x)) $, and the first part of the conclusion follows from \ref{range of the ap diff of the n.p.p. and dilations} and \ref{preparation for definition s.f.f.}. 

If $ 0 < s < r $ then $ a + su $ is a regular point of $\bm{\xi}_{A}$ by \ref{remark on regular points} and $\bm{\psi}_{A}(a+su) =(a,u) $. If $\tau \in T_{A}(a,u)$ and $ v \in \Real{n}$ is such that $\ap \Der \bm{\xi}_{A}(a+su)(v)=\tau$ then, noting that $ \ap \Der \bm{\xi}_{A}(a+su)(v) \bullet v \geq 0 $ by \ref{A lambda}\eqref{symmetry}, we use \ref{range of the ap diff of the n.p.p.} to compute
	\begin{flalign*}
		 Q_{A}(a,u)(\tau,\tau) & =	\ap \Der \bm{\xi}_{A}(a+su)(v) \bullet \ap \Der \bm{\nu}_{A}(a+su)(v)\\
		&  = s^{-1} \ap \Der \bm{\xi}_{A}(a+su)(v) \bullet (T_{\natural}(v) -\ap \Der \bm{\xi}_{A}(a+su)(v)  ) \\
		& = s^{-1} \ap \Der \bm{\xi}_{A}(a+su)(v) \bullet (v -\ap \Der \bm{\xi}_{A}(a+su)(v)  ) \\
		&  \geq -s^{-1}| \ap \Der \bm{\xi}_{A}(a+su)(v)|^{2} = -s^{-1}|\tau|^{2}.
	\end{flalign*}
Letting $ s \to r $ we get the second conclusion.
\end{proof}

\begin{Definition}\label{principal curvatures}
Let $ A \subseteq \Real{n} $ be closed. For each regular point $(a,u)$ of $N(A)$ we define \emph{the principal curvatures of $ A $ at $(a,u) $}, \begin{equation*}
\kappa_{A,1}(a,u) \leq \ldots \leq \kappa_{A,n-1}(a,u),
\end{equation*}
so that $\kappa_{A,m +1}(a,u) = \infty $, $\kappa_{A,1}(a,u), \ldots , \kappa_{A,m}(a,u)$ are the eigenvalues of $ Q_{A}(a,u)$ and $ m = \dim T_{A}(a,u)$. Moreover
\begin{equation*}
	\chi_{A,1}(x) \leq \ldots \leq \chi_{A,n-1}(x)
\end{equation*}
are the eigenvalues of $ \ap \Der \bm{\nu}_{A}(x)| \{v : v \bullet \bm{\nu}_{A}(x)=0  \} $ for $ x \in R(A) $.
\end{Definition}

Now we clarify the relation between the $ \kappa_{A,i} $'s and the $ \chi_{A,i} $'s.

\begin{Lemma}\label{representation of principal curvatures}
If $ A \subseteq \Real{n} $ is closed and $(a,u)\in R(N(A)) $ then
\begin{equation*}
	\kappa_{A,i}(a,u)= \frac{\chi_{A,i}(a+ru)}{1-r\chi_{A,i}(a+ru)} 
\end{equation*}
for $ 0 < r < \sup\{ s : \bm{\delta}_{A}(a+su) =s   \} $ and $ i = 1, \ldots , n-1 $.
\end{Lemma}

\begin{proof}
	If $ (a,u) \in R(N(A)) $ and $ 0 < r < \sup\{ s : \bm{\delta}_{A}(a+su) =s   \} $ let 
	\begin{equation*}
		T = \{ v : v \bullet \bm{\nu}_{A}(a+ru) =0 \}
	\end{equation*}
and let $\{ v_{1}, \ldots , v_{n-1} \} $ be an orthonormal basis of $ T $ such that
\begin{equation*}
\ap \Der \bm{\nu}_{A}(a+ru)(v_{i})= \chi_{A,i}(a+ru)v_{i} \quad \textrm{for $ i = 1, \ldots , n-1 $.}
\end{equation*}
 It follows from \ref{range of the ap diff of the n.p.p.} that
\begin{equation*}
	\ap \Der \bm{\xi}_{A}(a+ru)(v_{i}) = (1-r\chi_{A,i}(a+ru))v_{i} \quad \textrm{for $ i = 1, \ldots , n-1 $,}
\end{equation*}
whence we conclude from the definitions \ref{second fundamental form} and \ref{principal curvatures} that
\begin{equation*}
\chi_{A,i}(a+ru) =r^{-1} \quad\textrm{for $ i > \dim T_{A}(a,u) $, }
\end{equation*}
\begin{equation*}
Q_{A}(a,u)(v_{i},v_{j}) = \chi_{A,j}(a+ru)(1-r\chi_{A,j}(a+ru))^{-1}v_{i}\bullet v_{j} \quad \textrm{for $ i,j \leq \dim T_{A}(a,u)$,}
\end{equation*}
\begin{equation*}
\kappa_{A,i}(a,u) = \chi_{A,i}(a+ru)(1-r\chi_{A,i}(a+ru))^{-1} \quad \textrm{for $ 1 \leq i \leq n-1$.}
\end{equation*}
\end{proof}

It is immediate from the following lemma to conclude that the principal curvatures introduced in \cite{MR2031455} coincides with those introduced in \ref{principal curvatures}, see \ref{comparison with Hug et all.}.

\begin{Lemma}\label{basic facts about s.f.f.}
		Suppose $A\subseteq \Real{n}$ is closed and $\theta$ is $ \Haus{n-1}\restrict N(A)$ measurable and $ \Haus{n-1}\restrict N(A)$ almost positive function such that $ \theta \Haus{n-1} \restrict N(A) $ is a Radon measure over $ \Real{n} \times \mathbf{S}^{n-1} $. Let $ \psi = \theta \Haus{n-1} \restrict N(A) $.
	
	Then the following three statements hold.
		\begin{enumerate}
		\item \label{basic facts about s.f.f. 1} For $\Haus{n-1}$ a.e.\ $(a,u)\in N(A)$, $\Tan^{n-1}(\psi, (a,u))$ is a $ (n-1) $ dimensional plane contained in $\Tan^{n-1}(\Haus{n-1}\restrict N(A), (a,u))$ and there exist $u_{1}, \ldots , u_{n-1} \in \Real{n} $ such that $\{u_{1}, \ldots , u_{n-1}, u\} $ is an orthonormal basis of $\Real{n}$ and 
		\begin{equation*}
		\bigg\{	\bigg(\frac{1}{(1+\kappa_{A,i}(a,u)^{2})^{1/2}} u_{i}, \frac{\kappa_{A,i}(a,u)}{(1+\kappa_{A,i}(a,u)^{2})^{1/2}} u_{i} \bigg): \; 1 \leq i \leq n-1 \bigg\}
		\end{equation*}
	is an orthonormal basis of $\Tan^{n-1}(\psi,(a,u))$\footnote{If $\kappa_{A,i}(a,u) = \infty $ the corresponding vector equals $(0,u_{i})$.}.
		\item \label{basic facts about s.f.f. 2}  For $\Haus{n-1}$ a.e.\ $(a,u)\in N(A)$,	
		\begin{equation*}
		T_{A}(a,u) = \mathbf{p}[\Tan^{n-1}(\psi, (a,u))] \quad \textrm{and} \quad Q_{A}(a,u)(\tau,\tau_{1}) = \tau \bullet \sigma_{1}
		\end{equation*}
		whenever $\tau \in T_{A}(a,u)$, $\tau_{1} \in T_{A}(a,u) $ and $(\tau_{1}, \sigma_{1}) \in \Tan^{n-1}(\psi, (a,u))$.
		\item \label{basic facts about s.f.f. 3} For every $ (\Haus{n-1}\restrict N(A)) $ integrable $ \overline{\Real{}} $ valued function $ f $ on $ N(A) $,
		\begin{flalign*}
		& \int_{N(A)} f(a,u)\, \prod_{i=1}^{n-1}\frac{ | \kappa_{A,i}(a,u) |   }{(1+\kappa_{A,i}(a,u)^{2})^{1/2}} \, d\Haus{n-1}(a,u) \\
		& \quad = \int_{\mathbf{S}^{n-1}}\int_{\{a : (a,v) \in N(A)\} \times \{v\}}f \, d\Haus{0}\, d\Haus{n-1}v.
		\end{flalign*}
		\end{enumerate}
\end{Lemma}

\begin{proof}
	The first part of \eqref{basic facts about s.f.f. 1} directly follows from \ref{countably rectifiable sets and Radon measures} and \ref{bi-lipschitz parametrization of the unit normal bundle}. We fix now $ \lambda > 1 $. For $ r>0 $ let $ P_{r} $ be the set of $x \in R(A) \cap D(A_{\lambda}) \cap S(A,r)$ such that the following two conditions are satisfied:
	\begin{equation*}
	\ap\Tan(S_{\lambda}(A,r),x) = \Real{n}\cap \{ v : v \bullet \bm{\nu}_{A}(x)=0  \},
	\end{equation*}
	\begin{center}
$	\Tan^{n-1}(\Haus{n-1}\restrict \bm{\psi}_{A}[S_{\lambda}(A,r)], \bm{\psi}_{A}(x)) = \Tan^{n-1}(\psi, \bm{\psi}_{A}(x)) $ is an $ n-1 $ dimensional plane.
	\end{center}
If $ r>0$ and $ x \in P_{r}$ it follows from \ref{A lambda}\eqref{differentiable extension of psi}, \ref{remark on S lambda}, \ref{approx tangent cone and bilipschitz maps} and \ref{approx tangent cone and bilipschitz maps:remark} that
	\begin{equation*}
	 \ap \Der \bm{\psi}_{A}(x)[\ap \Tan( S_{\lambda}(A,r), x)] = \Tan^{n-1}(\Haus{n-1}\restrict \bm{\psi}_{A}[S_{\lambda}(A,r)], \bm{\psi}_{A}(x)),
	\end{equation*}
	\begin{equation*}
	\mathbf{p}[\Tan^{n-1}(\psi, \bm{\psi}_{A}(x))] = \im \ap \Der \bm{\xi}_{A}(x), 
	\end{equation*}
	\begin{equation*}
		Q_{A}(\bm{\psi}_{A}(x))(\tau, \tau_{1}) = \tau \bullet \sigma_{1} 
	\end{equation*}
for $ \tau, \tau_{1} \in T_{A}(\bm{\psi}_{A}(x)) $ and $ (\tau_{1}, \sigma_{1}) \in \Tan^{n-1}(\psi, \bm{\psi}_{A}(x)) $ and if $\{ v_{1}, \ldots , v_{n-1} \} $ is an orthonormal basis of $ \ap \Tan(S_{\lambda}(A,r),x) $ such that $\ap \Der \bm{\nu}_{A}(x)(v_{i})= \chi_{A,i}(x)v_{i}$ for $ i = 1, \ldots , n-1 $, then we can easily check using \ref{representation of principal curvatures} that
	\begin{equation*}
	\bigg\{\bigg(\frac{1}{(1+\kappa_{A,i}(\bm{\psi}_{A}(x))^{2})^{1/2}} v_{i}, \frac{ \kappa_{A,i}(\bm{\psi}_{A}(x))}{(1+\kappa_{A,i}(\bm{\psi}_{A}(x))^{2})^{1/2}}v_{i} \bigg) : \; 1 \leq i \leq n-1 \bigg\}
	\end{equation*}
	is an orthonormal basis of $\Tan^{n-1}(\Haus{n-1}\restrict \bm{\psi}_{A}[S_{\lambda}(A,r)],\bm{\psi}_{A}(x))$.
Noting that
	\begin{equation*}
	\Haus{n-1}(S_{\lambda}(A,r) \sim P_{r}) =0 \quad \textrm{and} \quad \Haus{n-1}(\bm{\psi}_{A}[S_{\lambda}(A,r)] \sim \bm{\psi}_{A}[P_{r}] ) =0
	\end{equation*}
for $ \Leb{1}$ a.e.\ $ r > 0 $ and \ref{bi-lipschitz parametrization of the unit normal bundle}, the second part of \eqref{basic facts about s.f.f. 1} and \eqref{basic facts about s.f.f. 2} follow.
	
	Finally, when $f$ is a nonnegative $ (\Haus{n-1}\restrict N(A)) $ measurable $\overline{\Real{}}$ valued function, we may apply \cite[3.2.22(3)]{MR0257325} with $W$, $Z$ and $f$ replaced by $\bm{\psi}_{A}[S_{\lambda}(A,r)]$, $\mathbf{S}^{n-1}$ and $\mathbf{q}|\bm{\psi}_{A}[S_{\lambda}(A,r)]  $ to conclude
\begin{flalign*}
	& \int_{\bm{\psi}_{A}[S_{\lambda}(A,r)]} f(a,u)\, \prod_{i=1}^{n-1}| \kappa_{A,i}(a,u) | (1+\kappa_{A,i}(a,u)^{2})^{-1/2} \, d\Haus{n-1}(a,u) \\
	& \quad = \int_{\mathbf{S}^{n-1}}\int_{\{a : (a,v) \in \bm{\psi}_{A}[S_{\lambda}(A,r)] \} \times \{v\}}f \, d\Haus{0}\, d\Haus{n-1}v
	\end{flalign*}
for $\Leb{1}$ a.e.\ $ r > 0 $ and \eqref{basic facts about s.f.f. 3} is a consequence of \ref{bi-lipschitz parametrization of the unit normal bundle} and \cite[2.4.7]{MR0257325}. The general case asserted in \eqref{basic facts about s.f.f. 3} is then a consequence of \cite[2.4.4]{MR0257325}.
\end{proof}

\begin{Remark}\label{comparison with Hug et all.}
It follows from \ref{basic facts about s.f.f.}\eqref{basic facts about s.f.f. 1} that the principal curvatures on $N(A)$ introduced in \cite[p.\ 244]{MR2031455} coincide on $\Haus{n-1}$ almost all of $N(A)$ with the principal curvatures introduced in \ref{principal curvatures}. 
\end{Remark}

\begin{Remark}\label{comparison with Fu}
	In case $\reach(A) > 0$, it follows from \ref{basic facts about s.f.f.}\eqref{basic facts about s.f.f. 2} that $Q_{A}$ coincides with the second fundamental form of $A$ introduced in \cite[4.5]{MR1021369} on $\Haus{n-1}$ almost all of $N(A)$. 
\end{Remark}

\begin{Remark}
	It is not difficult to check using \ref{basic facts about s.f.f.}\eqref{basic facts about s.f.f. 2} that \emph{if $ A $ and $ B $ are closed subsets of $ \Real{n} $ then 
		\begin{equation*}
		Q_{A}(a,u) = Q_{B}(a,u) \quad \textrm{for $ \Haus{n-1} $ a.e.\ $(a,u) \in N(A) \cap N(B) $.}
		\end{equation*}}
\end{Remark} 
\section{Stratification and support measures}\label{section: curvature and strata}

Recalling that $\bm{\xi}_{A}^{-1}\{a\}$ is a convex subset for every $ a \in A $, see \ref{Unp}, we introduce the following stratification.
	
\begin{Definition}\label{strata}
		Suppose $ A $ is a closed subset of $ \Real{n} $. For each $ 0 \leq m \leq n $, we define \textit{the $m$-th stratum of $A$} by\footnote{The dimension of a convex subset $ K $ of $\Real{n}$ is the dimension of the affine hull of $K $ and it is denoted by $\dim K $.}
		\begin{equation*}
			A^{(m)} = A \cap \{ a : \dim \bm{\xi}_{A}^{-1}\{a\} = n-m \}.
		\end{equation*}
	\end{Definition}

\begin{Remark}\label{characterization of the strata in terms of Normal bundle}
In \cite[4.1]{2017arXiv170309561M} the \textit{distance bundle of $A$} is defined as
\begin{equation*}
	\Dis(A) = (\Real{n}  \times \Real{n}) \cap \{  (a,v): a \in A, \; |v| = \bm{\delta}_{A}(a+v) \}
\end{equation*}
and $\Dis(A,a) = \{ v : (a,v)\in A  \}$ is a closed convex subset of $ \Nor(A,a) $ with $ 0 \in \Dis(A,a) $ for every $ a \in A $, see \cite[4.2]{2017arXiv170309561M}. One readily sees that
	\begin{equation*}
	N(A) = \{ (a, |v|^{-1}v) : 0 \neq v \in \Dis(A,a)  \}
	\end{equation*}
and it follows from \cite[4.4]{2017arXiv170309561M} that
\begin{equation*}
\dim \Dis(A,a) = \dim \bm{\xi}_{A}^{-1}\{a\} \quad  \textrm{whenever $ a \in A $},
\end{equation*}
	\begin{equation*}
A^{(m)} = A \cap \{ a : \dim \Dis(A,a) = n-m \}.
\end{equation*} 
It is proved in \cite[4.12]{2017arXiv170309561M} that $ A^{(m)} $ is a countably $ m $ rectifiable Borel set which can be $ \Haus{m} $ almost covered by the union of a countable family of $ m $ dimensional submanifolds of $ \Real{n} $ of class $ 2 $. Finally one may use Coarea formula to infer that
	\begin{equation*}
A^{(m)} = A \cap \{ a :   0 < \Haus{n-m-1}(N(A,a)) < \infty \} \quad \textrm{for $ m = 0, \ldots , n-1 $.}
\end{equation*} 
\end{Remark}

\begin{Lemma}\label{upper bound on the dimension of the tangent space}
	Suppose $A \subseteq \Real{n}$ is closed, $ 0 \leq m \leq n-1 $ is an integer and $ x \in \bm{\xi}_{A}^{-1}[A^{(m)}] $ such that $ \ap \liminf_{y \to x}\rho(A,y) \geq \rho(A,x)> 1 $ and $ \bm{\xi}_{A} $ is approximately differentiable at $ x $.
	
	Then $\dim \im \ap \Der \bm{\xi}_{A}(x) \leq m $.\ In particular, $\dim T_{A}(a,u) \leq m$ if $(a,u)$ is a regular point of $N(A)$ such that $a \in A^{(m)}$.
\end{Lemma}

\begin{proof}
	Let $a= \bm{\xi}_{A}(x)$, $1 < \lambda < \rho(A,x) $  and $C=\bm{\xi}_{A}^{-1}[\{a\}] \cap A_{\lambda}$. First we prove that $C$ is a convex subset of $\Real{n}$ and 
	\begin{equation*}
	\dim C  = \dim \bm{\xi}_{A}^{-1}\{a\}=n-m.
	\end{equation*}
	In fact, $C = \{ y : \bm{\delta}_{A}(a + \lambda(y-a)) = \lambda |y-a| \}$ by \ref{basic properties of rho} and \ref{n.p.p. along radial directions} and $C$ is convex by \cite[4.8(2)]{MR0110078}. Moreover, if $ U $ is the relative interior\footnote{The relative interior of a convex subset $K$ of $\Real{n}$ is the interior of $K$ relative to the affine hull of $K$.} of $\bm{\xi}_{A}^{-1}\{a\}$, then $\{y : a + \lambda (y-a) \in U\}$ is contained in $C$ and it is open relative to the affine hull of $\bm{\xi}_{A}^{-1}\{a\}$. Therefore $\dim C = \dim \bm{\xi}_{A}^{-1}\{a\}$. 
		
		By \ref{A lambda}\eqref{differentiable extension of psi}, let $ F : \Real{n} \rightarrow \Real{n} $ be an extension of $\bm{\xi}_{A}|A_{\lambda} $ that is differentiable at $x$ with $\Der F(x) = \ap \Der \bm{\xi}_{A}(x)$. Since $F(y) = a$ whenever $y \in C$, we conclude that $\Der F(x)(y-x) =0 $ whenever $ y \in C $. Therefore $\Der F(x)(y-x)=0 $ whenever $ y $ belongs to the affine hull of $ C $. Since $\dim C = n-m $, we conclude 
		\begin{equation*}
		\dim \im \ap \Der \bm{\xi}_{A}(x) \leq m.
		\end{equation*} 
\end{proof}

We point out a Coarea-type formula for the generalized normal bundle.

\begin{Lemma}\label{Area formula Gauss map}
	If $ A \subseteq \Real{n} $ is closed set, $f$ is a $ (\Haus{n-1}\restrict N(A)) $ integrable $ \overline{\Real{}} $ valued function and $ 0 \leq m \leq n-1 $ then
			\begin{flalign*}
			& \int_{N(A)|A^{(m)}} f(a,u)\, \prod_{i=1}^{m}\frac{1}{(1+\kappa_{A,i}(a,u)^{2})^{1/2}} \, d\Haus{n-1}(a,u) \\
			& \quad = \int_{A^{(m)}}\int_{\{z\} \times N(A,z)}f \, d\Haus{n-m-1}\, d\Haus{m}z.
			\end{flalign*}
	\end{Lemma}

\begin{proof}
We assume $ f \geq 0 $ on $\Haus{n-1}$ almost all of $N(A)$, since, as usual, the general case follows from \cite[2.4.4]{MR0257325}. Since $A^{(0)}$ is a countable set by \ref{characterization of the strata in terms of Normal bundle}, the case $ m =0 $ is clear. Therefore we assume $ m \geq 1 $, we let $ \lambda > 1 $ and we define $ C_{i} = \bm{\psi}_{A}[S_{\lambda}(A,1/i)] $ for every integer $ i \geq 1 $. Since $\kappa_{A,m+1}(a,u) = \infty $ for $\Haus{n-1}$ a.e.\ $(a,u)\in N(A)|A^{(m)}$ by \ref{upper bound on the dimension of the tangent space}, noting \ref{remark on S lambda}, the conclusion can be easily derived in two simple steps: first we apply Coarea formula \cite[p.\ 300]{MR0467473} with $ W $, $ f $ and $ S $ replaced by $ C_{i} $, $ \mathbf{p}|C_{i} $ and $A^{(m)}$ respectively, second we let $ i \to \infty $ and we recall \ref{bi-lipschitz parametrization of the unit normal bundle}.
\end{proof}

\begin{Remark}
	If $ \reach(A)> 0 $ and $ f $ is the characteristic funtion of a Borel subset of $ N(A) $ then the conclusion of \ref{Area formula Gauss map} is essentially contained in \cite[3.2]{MR1652084}.
\end{Remark}

\begin{Remark}\label{Area formula Gauss map:remark}
	The following corollary can be deduced from \ref{Area formula Gauss map}. \emph{If $ S \subseteq A $ and $ 1 \leq m \leq n-1 $ then $ \Haus{m}(S \cap A^{(m)}) =0 $ if and only if 
	\begin{equation*}
	\kappa_{A,m}(a,u) = \infty \quad \textrm{for $ \Haus{n-1} $ a.e.\ $ (a,u) \in N(A)|S \cap A^{(m)}$}.
	\end{equation*}}
\end{Remark}

We obtain here an integral representation for the support measures.

\begin{Theorem}\label{representation of support measures}
	Suppose $ A \subseteq \Real{n} $ is a closed set, $ \mu_{0}, \ldots , \mu_{n-1} $ are the support measures of $ A $, $ 1 \leq m \leq n-1 $ is an integer and $ S $ is a countable union of Borel subsets with finite $ \Haus{m} $ measure.
	
	Then the following two statements hold.
\begin{enumerate}
	\item \label{representation of support measures:2} If $ j > m $ then $ \kappa_{A,m}(x,u) = \infty $ for $ \Haus{n-1} $ a.e.\ $(x,u) \in N(A)|S \cap A^{(j)} $;
	\item \label{representation of support measures:3} if $ T \subseteq N(A)|S $ is $ \Haus{n-1} $ measurable then
		\begin{equation*}
	\mu_{m}(T) = \frac{1}{(n-m)\bm{\alpha}(n-m)}\int \Haus{n-m-1}\{v : (z,v) \in T \}d\Haus{m}z.
	\end{equation*}
\end{enumerate}
\end{Theorem}

\begin{proof}
	Suppose $ S_{1},S_{2}, \ldots $ is a sequence of Borel subsets with finite $ \Haus{m} $ measure whose union equals $ S $ and $ S_{i} \subseteq S_{i+1} $ for $ i \geq 1 $. Let $ \lambda > 1 $ and $ C_{i} = \bm{\psi}_{A}[S_{\lambda}(A, 1/i)] $. We apply the co-area formula in \cite[p.\ 300]{MR0467473} with $ W $, $ f $ and $ S $ replaced by $ C_{i} $, $ \mathbf{p}|C_{i} $ and $ S_{i} \cap A^{(j)} $ to infer that
	\begin{equation*}
	\int_{C_{i}|S_{i}\cap A^{(j)}}\|\textstyle \bigwedge_{m}[ \mathbf{p}|\Tan^{n-1}(\Haus{n-1}\restrict C_{i}, (x,u))] \| \, d\Haus{n-1}(x,u) =0
	\end{equation*}
whenever $ j > m $. It follows that
\begin{equation*}
\dim \mathbf{p}[\Tan^{n-1}(\Haus{n-1}\restrict C_{i}, (x,u))] < m,
\end{equation*}
whence we deduce that $ \kappa_{A,m}(x,u) = \infty $ for $ \Haus{n-1} $ a.e.\ $ (x,u) \in C_{i}|S_{i} \cap A^{(j)} $ and for $ j > m $ by \ref{basic facts about s.f.f.}\eqref{basic facts about s.f.f. 2}. Then we obtain \eqref{representation of support measures:2} letting $ i \to \infty $ and noting \ref{bi-lipschitz parametrization of the unit normal bundle}.

Since $ \kappa_{A,m}(x,u) = \infty $ for $ \Haus{n-1} $ a.e.\ $ (x,u) \in N(A)|A^{(j)} $ if $ j < m $ by \ref{upper bound on the dimension of the tangent space}, we conclude from \eqref{symmetric functions} that
\begin{equation*}
	H_{n-m-1}(x,u) =0 \quad \textrm{for $ \Haus{n-1} $ a.e.\ $ (x,u) \in N(A)|S \cap A^{(j)} $},
\end{equation*}
if $ j \neq m $. Since $ \kappa_{A,m+1}(x,u) = \infty $ for $ \Haus{n-1} $ a.e.\ $ (x,u) \in N(A)|A^{(m)} $ by \ref{upper bound on the dimension of the tangent space}, it follows that
\begin{equation*}
	H_{n-m-1}(x,u) = \prod_{i=1}^{m}\frac{1}{(1+\kappa_{A,i}(x,u)^{2})^{1/2}} \quad \textrm{for  $ \Haus{n-1} $ a.e.\ $ (x,u) \in N(A)|A^{(m)} $.}
\end{equation*}
Then \eqref{representation of support measures:3} follows from \ref{Area formula Gauss map}.
\end{proof}

\begin{Remark}
	The integral representation in \ref{representation of support measures}\eqref{representation of support measures:3} has been proved in \cite[5.5]{MR1742247} for sets of positive reach.
\end{Remark}

\begin{Remark}
Since $ A^{(n-1)} $ is countably $(n-1)$ rectifiable and $ \Haus{n-1}(A^{(i)}) =0 $ for $ i < n-1 $ (see \ref{characterization of the strata in terms of Normal bundle}) it follows from \ref{representation of support measures} that \emph{if $ T \subseteq N(A) $ is $ \Haus{n-1} $ measurable then
	\begin{equation*}
	\mu_{n-1}(T) = \frac{1}{2}\int \Haus{0}\{ v : (z,v)\in T  \}d\Haus{n-1}z.
	\end{equation*}}
This formula is equivalent to \cite[4.1]{MR2031455}.
\end{Remark}

\section{Relation with second order rectifiability}\label{section: relation with second order rectifiability}

In this final section we prove that, in a certain sense, the "absolutely continuous part" of the second fundamental form introduced in section \ref{section: s.f.f.} can be described by the approximate differential of order $ 2 $ introduced by the author in \cite{2017arXiv170107286S}, see \ref{agreement with classical second fundamental form}.

\begin{Lemma}\label{agreement with classical second fundamental form: lemma}
	Suppose $ A \subseteq \Real{n} $ is closed, $ 1 \leq m \leq n-1 $ and let $ M $ be an $ m $ dimensional submanifold of class $ 2 $.
	
	Then there exists $ R \subseteq A \cap M $ such that $ \Haus{m}((A \cap M) \sim R) =0 $ and 
	\begin{equation*}
	Q_{A}(a,u) = - \mathbf{b}_{M}(a) \bullet u \quad \textrm{for $ \Haus{n-1} $ a.e.\ $ (a,u) \in N(A)|R $.}
	\end{equation*}
\end{Lemma}

\begin{proof}
We define 
\begin{equation*}
 R = A \cap M \cap \{ a : N(A,a) \subseteq N(M,a)   \}.
\end{equation*}
Since $\Hdensity{m}{M \sim A}{a} =0$ for  $ \Haus{m} $ a.e.\ $ a \in A \cap M $ by \cite[2.10.19(4)]{MR0257325}, it follows from \cite[3.2.16]{MR0257325} that
	\begin{equation*}
	\Tan(M,a)= \Tan^{m}(\Haus{m}\restrict A \cap M,a) \subseteq \Tan(A,a)
	\end{equation*}
for $ \Haus{m} $ a.e.\ $ a \in A \cap M $. Henceforth, we deduce from \ref{Comparison with HLW04 and MS17} that $	N(A,a) \subseteq N(M,a) $ for $ \Haus{m} $ a.e.\ $ a \in A \cap M $ and
\begin{equation*}
	 \Haus{m}((A \cap M)\sim R) =0.
\end{equation*}
	
	We recall that $N(M)$ is an $n-1 $ dimensional submanifold of class 1 in $ \Real{n} \times \mathbf{S}^{n-1}$, see \ref{normal bundle of smooth submanifolds}. Since $N(A)|R \subseteq N(M)$, we may use \cite[2.10.19(4), 3.2.16]{MR0257325} to get $ \Hdensity{n-1}{N(M) \sim N(A)|R}{(a,u)} =0 $ and
	\begin{equation*}
	 \Tan(N(M), (a,u)) = \Tan^{n-1}(\Haus{n-1}\restrict N(A)|R, (a,u)) 
	\end{equation*}
	for $ \Haus{n-1} $ a.e.\ $ (a,u) \in N(A)|R $. Finally, if $ \psi $ is a Radon measure as in \ref{basic facts about s.f.f.}, we combine \ref{basic facts about s.f.f.}\eqref{basic facts about s.f.f. 1} and \ref{countably rectifiable sets and Radon measures} to deduce that 
	\begin{equation*}
	\Tan(N(M),(a,u)) = \Tan^{n-1}(\psi, (a,u)) \quad \textrm{for $ \Haus{n-1} $ a.e.\ $ (a,u) \in N(A)|R $}.
	\end{equation*}
Now we conclude employing \ref{basic facts about s.f.f.}\eqref{basic facts about s.f.f. 2} and \ref{remark on s.f.f for smooth submanifolds}.
\end{proof}

\begin{Theorem}\label{agreement with classical second fundamental form}
	Let $ A \subseteq \Real{n} $ be a closed set, $ 1 \leq m \leq n-1 $  and let $ S \subseteq A $ be $ \Haus{m} $ measurable and $ \rect{m} $ rectifiable of class $ 2 $. Then there exists $ R \subseteq S $ such that $ \Haus{m}(S \sim R) = 0 $ and\footnote{If $ f : V \rightarrow W $ is a linear map between vector spaces then $\bigodot_{2}f : V \times V \rightarrow W \times W $ is defined by $ \bigodot_{2}f(u,v) =(f(u), f(v)) $ for $ (u,v) \in V \times V $. Note that this notation does not agree with \cite[1.9.1]{MR0257325}} 
	\begin{equation*}
	\textstyle	\ap \Tan(S,a) = T_{A}(a,u) \quad \ap \Der^{2}S(a) \bullet u = - Q_{A}(a,u) \circ \bigodot_{2} T_{A}(a,u)_{\natural}
	\end{equation*}
	for $ \Haus{n-1} $ a.e.\ $ (a,u) \in N(A)|R $.
\end{Theorem}

\begin{proof}
	If $ \{ M_{i}: i \geq 1 \} $ is a sequence of $ m $ dimensional submanifolds of class $ 2 $ covering $ \Haus{m} $ almost all of $ S $, we apply \cite[2.10.19(4)]{MR0257325} and \cite[3.22]{2017arXiv170107286S} to obtain
	\begin{equation*}
	\ap \Tan(S,a) = \Tan(M_{i},a), \quad	\ap \Der^{2}S(a) = \textstyle \mathbf{b}_{M_{i}}(a) \circ \bigodot_{2}\Tan(M_{i},a)_{\natural},
	\end{equation*}
	for $ \Haus{m} $ a.e.\ $ a \in M_{i} \cap S $ and for every $ i \geq 1 $. Now the conclusion easily follows applying \ref{agreement with classical second fundamental form: lemma}.
\end{proof}

The following lemma shows that the approximate differential of order $ 2 $ of a second order rectifiable closed set $ S \subseteq \Real{n} $ does not always fully describe its second fundamental form $ Q_{S} $. The same phenomenon arises in the theory of functions of bounded variation: the total differential is not always fully described by the approximate gradient. It seems to be not a coincidence that the following example considers exactly the primitive of a function of bounded variation whose total differential cannot be fully described by the approximate derivative. Recall that the boundary of a convex set of $ \Real{n} $ is always countably $ \rect{n-1} $ rectifiable of class $ 2 $.

\begin{Lemma}\label{Cantor function}
There exists a closed convex set $ A \subseteq \Real{2} $ and a subset $ T $ of the topological boundary of $ A $ such that \mbox{$ \Haus{1}(T)=0 $,} $ \Haus{1}(N(A)|T)>0 $ and
	\begin{equation*}
	T_{A}(a,u) = \{0\} \quad \textrm{for $ \Haus{1} $ a.e.\ $ (a,u) \in N(A)|T $.}
	\end{equation*}
\end{Lemma}
\begin{proof}
Let $ 0 < s < 1 $ and let $C \subseteq \Real{}$ be a compact set with $0 <\Haus{s}(C) < \infty$. Define
\begin{equation*}
f(x) = \Haus{s}(C \cap \{z : z \leq x \}) \quad \textrm{for $ x \in \Real{} $,}
\end{equation*}
and let $ g $ be a primitive of $ f $. Then $ g $ is a non-decreasing convex function of class $1$ on $ \Real{}$ and we define
\begin{equation*}
A = \Real{2} \cap \{ (x,y) : g(x) \leq y \}, \quad T = \{ (x,g(x)) : x \in C \}.
\end{equation*}
We notice that $A$ is a closed convex set, $ T \subseteq  A^{(1)} $, $\Haus{1}(T) = 0$ and
\begin{equation*}
N(A,(x,g(x))) = \{ (1+f(x)^{2})^{-1/2}(f(x),-1) \} \quad \textrm{whenever $ x \in \Real{} $.}
\end{equation*}
It follows that $\Haus{1}(\mathbf{q}(N(A))) > 0 $. Moreover, since $ f $ is constant on each connected component of $ \Real{} \sim C $, it follows that $ \mathbf{q}(N(A)|A \sim T) $ is a countable subset of $ \mathbf{S}^{1} $; in particular $  \Haus{1}(\mathbf{q}(N(A)|A \sim T)) =0 $. Therefore one easily infers that
\begin{equation*}
\Haus{1}(N(A)|T) > 0.
\end{equation*}
Finally we notice that $T_{A}(a,u) = \{0\} $ for $\Haus{1}$ a.e.\ $ (a,u) \in N(A)|T$ by \ref{Area formula Gauss map:remark}.
\end{proof}

\begin{Remark}
If $ M $ is an $ m $ dimensional submanifold of class $ 1 $ in $ \Real{n} $ that meets every $ m $ dimensional submanifold of class $ 2 $ in a set of $ \Haus{m} $ measure zero then it follows from \cite[4.12]{2017arXiv170309561M} that $ \Haus{m}(M^{(m)}) =0 $. Since $ M^{(i)} = \varnothing $ if $ i < m $ by \ref{Comparison with HLW04 and MS17}, it follows from  \ref{Area formula Gauss map:remark} and \ref{representation of support measures}\eqref{representation of support measures:2} that
\begin{equation*}
\dim T_{M}(a,u) \leq m-1 \quad \textrm{for $ \Haus{n-1} $ a.e.\ $(a,u) \in N(M) $.}
\end{equation*}
The existence of such $ M $ can be inferred from \cite{MR0427559}. 
\end{Remark}

\appendix

\section*{Appendix}

In this appendix we collect for the reader's convenience some remarks that are simple consequences of known facts.

\section{On approximate differentiability}\label{section: On approximate differentiability}

Basic facts on approximate differentiability for functions are collected in \cite[\S 2]{2017arXiv170107286S}. Here we point out some additional remarks.

\begin{Lemma}\label{local maximum of approximately differentiable functions}
	Suppose $ n \geq 1 $ is an integer, $ B \subseteq A \subseteq \Real{n}$, $ a \in A$ and \mbox{$ f: A \rightarrow \Real{} $} are such that $f$ is approximately differentiable at $ a $, $\supLdensity{n}{B}{a}=1$ and $ f(x) \leq f(a)$ for every $ x \in B$.
	
	Then $ \ap \Der f(a) =0 $.
\end{Lemma}
\begin{proof}
	Assume $ a=0 $ and $ f(0)=0 $. If $ \ap\Der f(0)\neq 0 $ then there would be $ \epsilon >0 $ and a non empty open cone $ C $ such that $\ap\Der f(0)(x) \geq 2\epsilon |x| $ for every $ x \in C $. Therefore $ f(x)- \ap \Der f(0)(x)  \leq - 2\epsilon |x| $ for every $ x \in C \cap B $ and
	\begin{eqnarray*}
		&	\supLdensity{n}{B \sim C}{0}<1, \quad \supLdensity{n}{B\cap C}{0}>0, &\\
		& \supLdensity{n}{\Real{n}\sim \{x: | f(x)- \ap\Der f(0)(x) | \leq \epsilon |x|   \}}{0}>0. &
	\end{eqnarray*}
	This would be a contradiction.
\end{proof}

\begin{Remark}
	We observe that a similar argument proves that \textit{if $ f $ is approximately differentiable of order $ 2 $ at $ a $ then $ \ap \Der^{2} f(a) \leq 0 $.}
\end{Remark}

\begin{Lemma}\label{bi-Lip and inverse of ap Df}
	Suppose $n \geq 1 $ and $\nu \geq 1 $ are integers, $ B \subseteq A \subseteq \Real{n}$, $ a \in B$ and $ f: A \rightarrow \Real{\nu} $ are such that $f$ is approximately differentiable at $ a $, $f|B$ is a bi-Lipschitzian homeomorphism and $\Ldensity{n}{\Real{n} \sim B}{a}=0$. 
	
	Then $\ker \ap \Der f(a) = \{0\}$.
\end{Lemma}

\begin{proof}
	If $\Gamma = (1/2)(\Lip(f|B)^{-1})^{-1}$ then $ | f(y) - f(x)| \geq 2\Gamma |y-x|$ whenever $ y,x \in B $. If there was $ v \in \Real{n}\sim \{0\}$ such that $\ap \Der f(a)(v) =0 $, then there would exist a non empty open cone $C$ such that 
	\begin{equation*}
	|	\ap \Der f(a)(u)| \leq \Gamma |u| \quad \textrm{whenever $ u \in C$.}
	\end{equation*}
	Choosing $ 0 < \epsilon < \Gamma $ and letting $ D = \{ u + a : u \in C\}$ and 
	\begin{equation*}
	E = A \cap \{ x : |f(x) -f(a)- \ap \Der f(a)(x-a)| \leq \epsilon |x-a| \},
	\end{equation*}
	we would notice that $\Ldensity{n}{\Real{n}\sim E}{a} =0 $ and $ B \cap D \cap E = \varnothing $ and we would get a contradiction.
\end{proof}

\begin{Lemma}\label{composition and approx diff}
	If $ m,n,\nu $ are positive integers, $ D \subseteq \Real{m} $, $ U \subseteq \Real{n} $ is open, \mbox{$ f : D \rightarrow \Real{n} $,} $ g : U \rightarrow \Real{\nu} $, $ x \in D $, $ f(x) \in U $, $ f $ is approximately differentiable at $ x $ and $ g $ is differentiable at $ f(x) $, then $ g \circ f $ is approximately differentiable at $ x $ with
	\begin{equation*}
	\ap \Der(g \circ f)(x) = \Der g(f(x)) \circ \ap \Der f(x).
	\end{equation*}
\end{Lemma}

\begin{proof}
	Combine \cite[2.8]{2017arXiv170107286S} and \cite[3.1.1(2)]{MR0257325}.
\end{proof}

\begin{Lemma}\label{approx. versus pointwise diff. for Lip functions}
	If $ n, \nu \geq 1 $ are integers, $ D \subseteq \Real{n} $, $ z \in D $ and $ g : \Real{n} \rightarrow \Real{\nu} $ is a Lipschitzian function such that $ g|D $ is approximately differentiable at $ z $, then $ g $ is differentiable at $ z $ with $ \ap \Der (g|D)(z) = \Der g(z) $.
\end{Lemma}

\begin{proof}
	This is proved in \cite[3.1.5]{MR0257325}.
\end{proof}

\section{On the tangent cone of a measure}

The concept of approximate tangent vector to a measure is introduced in \cite[3.2.16]{MR0257325}. Besides the fundamental results given in \cite[3.2.16--3.2.22, 3.3.18]{MR0257325}, we point out here some useful consequences.

First, the following elementary inequality is useful here and elsewhere.
\begin{Lemma}\label{integral and bi-lip functions}
	If $X$ and $Y$ are metric spaces, $ m \geq 1 $ is an integer, \mbox{$ \theta(x) \geq 0 $} for $\Haus{m}$ a.e.\ $ x \in X$, $ 0 \leq \gamma < \infty $ and $ f : X \rightarrow Y $ is an univalent Lipschitzian map onto $Y$ such that $\gamma$ is a Lipschitz constant for $f^{-1}$, then
	\begin{equation*}
	\int^{\ast}_{X}\theta d\Haus{m} \leq \gamma^{m} \int_{Y}^{\ast}\theta \circ f^{-1}d\Haus{m}.
	\end{equation*}
\end{Lemma}

\begin{proof}
	We assume $\int_{Y}^{\ast}\theta \circ f^{-1}d\Haus{m}< \infty$. Then the conclusion easily follows from the definition of upper integral in \cite[2.4.2]{MR0257325}, using approximation by upper functions.
\end{proof}

\begin{Lemma}\label{approx tangent cone and bilipschitz maps}
	Suppose $X$ and $Y$ are normed vector spaces, $ P \subseteq X$, $ m \geq 1 $ is an integer, $ \theta(x) \geq 0 $ for $\Haus{m}$ a.e.\ $ x \in P$, $ a \in P$ and $ f : X \rightarrow Y $ is a function differentiable at $a$ such that $ f|P$ is a bi-Lipschitzian homeomorphism. Additionally, we define the measures
	\begin{equation*}
	\psi =  \theta \Haus{m}\restrict P, \quad \mu = (\theta \circ (f|P)^{-1} )\Haus{m} \restrict f(P).
	\end{equation*}

	Then $\Der f(a)[\Tan^{m}(\psi,a)] \subseteq \Tan^{m}(\mu, f(a))$.
\end{Lemma}
\begin{proof}
	Firstly we prove that $ \bm{\Theta}^{m}(\psi \restrict X \sim f^{-1}[T], a) =0 $, whenever $ T \subseteq Y$ such that $\bm{\Theta}^{m}( \mu \restrict Y \sim T,f(a)) =0 $. In fact, for such a subset $T$, if $S = f^{-1}[T]$, $\gamma$ is a Lipschitz constant for $f|P$ and$(f|P)^{-1}$ and $ r > 0 $, we observe that
	\begin{equation*}
	f[(P \sim S) \cap \mathbf{B}(a,r)] \subseteq (f[P] \sim T) \cap \mathbf{B}(f(a), \gamma r),
	\end{equation*}
	and we employ \ref{integral and bi-lip functions} to get that $\psi(\mathbf{B}(a,r) \sim S)\leq \gamma^{m}\mu(\mathbf{B}(f(a), \gamma r) \sim T) $. Therefore $\Der f(a)[\Tan^{m}(\psi,a)] \subseteq \Tan^{m}(\mu, f(a))$ by \cite[3.1.21, p.\ 234]{MR0257325} and \cite[3.2.16, p.\ 252]{MR0257325}.
\end{proof}

\begin{Remark}\label{approx tangent cone and bilipschitz maps:remark}
	If $ \theta $ is the characteristic function of $P$ then, by \cite[2.4.5]{MR0257325}, we have that $\psi = \Haus{m}\restrict P$ and $\mu =\Haus{m}\restrict f[P]$.
\end{Remark}

\begin{Lemma}\label{countably rectifiable sets and Radon measures}
	Suppose $ 1 \leq k \leq \nu $ are integers, $ E \subseteq \Real{\nu} $ is countably $\rect{k}$ rectifiable and $\Haus{k}$ measurable and $\theta$ is a $\Haus{k}\restrict E$ measurable $ \Haus{k}\restrict E $ almost positive function such that
\begin{equation*}
	\psi = \theta \Haus{k} \restrict E 
	\end{equation*}
is a Radon measure over $ \mathbf{R}^{\nu} $. 

	Then $\Tan^{k}(\psi,z)$ is a $k$ dimensional plane contained in $\Tan^{k}(\Haus{k} \restrict E,z)$ for $\Haus{k}$ a.e.\ $ z \in E$ and 
	\begin{equation*}
	\Tan^{k}(\Haus{k}\restrict F,z) \subseteq \Tan^{k}(\psi, z)  \quad \textrm{for $\Haus{k}$ a.e.\ $ z \in F$,}
	\end{equation*}
	whenever $ F \subseteq E $ is $\Haus{k}$ measurable such that $\Haus{k}(F) < \infty $. 
\end{Lemma}
\begin{proof}
	Firstly we observe that $\psi(S) = 0 $ if and only if $\Haus{k}(S) =0$. Therefore $\Real{\nu}$ is $(\psi,k)$ rectifiable and, employing \cite[2.4.10, 2.10.19(3)]{MR0257325},
	\begin{equation*}
	\bm{\Theta}^{\ast k}(\psi,z)< \infty \quad \textrm{for $\psi$ a.e.\ $ z \in \Real{\nu} $.}
	\end{equation*}
	We apply \cite[3.3.18]{MR0257325} to conclude that $\Tan^{k}(\psi,z) \in \mathbf{G}(n,k)$ for $\Haus{k}$ a.e.\ $ z \in E$. If $ F \subseteq E $ is $\Haus{k}$ measurable and $\Haus{k}(F) < \infty $, we define
	\begin{equation*}
	F_{i} = F \cap \{ z : \theta(z)\geq i^{-1}  \} \quad \textrm{for every integer $ i \geq 1 $,}
	\end{equation*}
	we observe that $\Tan^{k}(\Haus{k}\restrict F,z)= \Tan^{k}(\Haus{k}\restrict F_{i},z)$ for $\Haus{k}$ a.e.\ $ z \in F_{i}$ by \cite[2.10.19(4)]{MR0257325}, and we use \cite[3.2.16]{MR0257325} to conclude 
	\begin{equation*}
	\Tan^{k}(\Haus{k}\restrict F, z) \subseteq \Tan^{k}(\psi, z) \quad \textrm{for $\Haus{k}$ a.e.\ $ z \in F $.}
	\end{equation*}
	Since by \cite[3.2.14]{MR0257325} the set $E$ can be $\Haus{k}$ almost covered by countably many $\Haus{k}$ measurable $k$ rectifiable subsets of $\Real{\nu}$, we may apply \cite[3.2.19]{MR0257325} to conclude that $\Tan^{k}(\psi, z) \subseteq \Tan^{k}(\Haus{k}\restrict E, z)$ for $\Haus{k}$ a.e.\ $ z \in E$.
\end{proof}


\begin{thebibliography}{{San}17}
	
	\bibitem[Asp73]{MR0310150}
	Edgar Asplund.
	\newblock Differentiability of the metric projection in finite-dimensional
	{E}uclidean space.
	\newblock {\em Proc. Amer. Math. Soc.}, 38:218--219, 1973.
	
	\bibitem[Buc92]{MR1100645}
	Zolt{\'a}n Buczolich.
	\newblock Density points and bi-{L}ipschitz functions in {${\bf R}^m$}.
	\newblock {\em Proc. Amer. Math. Soc.}, 116(1):53--59, 1992.
	
	\bibitem[CH00]{MR1742247}
	Andrea Colesanti and Daniel Hug.
	\newblock Hessian measures of semi-convex functions and applications to support
	measures of convex bodies.
	\newblock {\em Manuscripta Math.}, 101(2):209--238, 2000.
	
	\bibitem[Fed59]{MR0110078}
	Herbert Federer.
	\newblock Curvature measures.
	\newblock {\em Trans. Amer. Math. Soc.}, 93:418--491, 1959.
	
	\bibitem[Fed69]{MR0257325}
	Herbert Federer.
	\newblock {\em Geometric measure theory}.
	\newblock Die Grundlehren der mathematischen Wissenschaften, Band 153.
	Springer-Verlag New York Inc., New York, 1969.
	
	\bibitem[Fed78]{MR0467473}
	Herbert Federer.
	\newblock Colloquium lectures on geometric measure theory.
	\newblock {\em Bull. Amer. Math. Soc.}, 84(3):291--338, 1978.
	
	\bibitem[Fer76]{MR0413112}
	Steve Ferry.
	\newblock When {$\epsilon $}-boundaries are manifolds.
	\newblock {\em Fund. Math.}, 90(3):199--210, 1975/76.
	
	\bibitem[Fu85]{MR816398}
	Joseph Howland~Guthrie Fu.
	\newblock Tubular neighborhoods in {E}uclidean spaces.
	\newblock {\em Duke Math. J.}, 52(4):1025--1046, 1985.
	
	\bibitem[Fu89]{MR1021369}
	Joseph H.~G. Fu.
	\newblock Curvature measures and generalized {M}orse theory.
	\newblock {\em J. Differential Geom.}, 30(3):619--642, 1989.
	
	\bibitem[GP72]{MR0287442}
	Ronald Gariepy and W.~D. Pepe.
	\newblock On the level sets of a distance function in a {M}inkowski space.
	\newblock {\em Proc. Amer. Math. Soc.}, 31:255--259, 1972.
	
	\bibitem[HLW04]{MR2031455}
	Daniel Hug, G{\"u}nter Last, and Wolfgang Weil.
	\newblock A local {S}teiner-type formula for general closed sets and
	applications.
	\newblock {\em Math. Z.}, 246(1-2):237--272, 2004.
	
	\bibitem[Hug98]{MR1652084}
	Daniel Hug.
	\newblock Generalized curvature measures and singularities of sets with
	positive reach.
	\newblock {\em Forum Math.}, 10(6):699--728, 1998.
	
	\bibitem[Koh77]{MR0427559}
	Robert~V. Kohn.
	\newblock An example concerning approximate differentiation.
	\newblock {\em Indiana Univ. Math. J.}, 26(2):393--397, 1977.
	
	\bibitem[KW14]{MR3153586}
	E.~V. Khmaladze and W.~Weil.
	\newblock Differentiation of sets---the general case.
	\newblock {\em J. Math. Anal. Appl.}, 413(1):291--310, 2014.
	
	\bibitem[Las06]{MR2307063}
	G\"{u}nter Last.
	\newblock On mean curvature functions of {B}rownian paths.
	\newblock {\em Stochastic Process. Appl.}, 116(12):1876--1891, 2006.
	
	\bibitem[MS17]{2017arXiv170309561M}
	U.~{Menne} and M.~{Santilli}.
	\newblock {A geometric second-order-rectifiable stratification for closed
		subsets of Euclidean space}.
	\newblock {\em ArXiv e-prints}, March 2017. \newblock{To appear in  \em Ann. Sc. Norm. Super. Pisa Cl. Sci. (5).}
	
	\bibitem[RW10]{MR2865426}
	Jan Rataj and Steffen Winter.
	\newblock On volume and surface area of parallel sets.
	\newblock {\em Indiana Univ. Math. J.}, 59(5):1661--1685, 2010.
	
	\bibitem[RZ12]{MR2954647}
	Jan Rataj and Lud\v{e}k Zaj\'{i}\v{c}ek.
	\newblock Critical values and level sets of distance functions in {R}iemannian,
	{A}lexandrov and {M}inkowski spaces.
	\newblock {\em Houston J. Math.}, 38(2):445--467, 2012.
	
	\bibitem[{San}17]{2017arXiv170107286S}
	M.~{Santilli}.
	\newblock {Rectifiability and approximate differentiability of higher order for
		sets}.
	\newblock {\em ArXiv e-prints}, January 2017. \newblock{To appear in \em Indiana Univ. Math. J.}
	
	\bibitem[Sta79]{MR534512}
	L.~L. Stach\'o.
	\newblock On curvature measures.
	\newblock {\em Acta Sci. Math. (Szeged)}, 41(1-2):191--207, 1979.
	
	\bibitem[Wal76]{MR0417984}
	Rolf Walter.
	\newblock Some analytical properties of geodesically convex sets.
	\newblock {\em Abh. Math. Sem. Univ. Hamburg}, 45:263--282, 1976.
	
	\bibitem[Z{\"a}h86]{MR849863}
	M.~Z{\"a}hle.
	\newblock Integral and current representation of {F}ederer's curvature
	measures.
	\newblock {\em Arch. Math. (Basel)}, 46(6):557--567, 1986.
	
\end{thebibliography}

\end{document}